\documentclass[a4paper]{article}
\usepackage[utf8]{inputenc}
\usepackage{graphicx} 
\usepackage{graphicx} 
\usepackage{hyperref}
\usepackage{amssymb}
\usepackage{amsthm,verbatim}
\usepackage{lipsum}
\usepackage{amsfonts}
\usepackage{graphicx}
\usepackage{epstopdf}
\usepackage{algorithm}
\usepackage{algorithmic}
\usepackage{amsmath}
\usepackage{verbatim}
\usepackage{graphics}
\usepackage{graphicx}
\usepackage{subfigure}
\usepackage{bm}
\usepackage{multirow,multicol}
\usepackage{float}
\usepackage{appendix}
\usepackage{chngcntr}
\usepackage{color}
\usepackage{tikz}
\usepackage{pgfplots}
\usepackage{booktabs}
\usepackage{subfigure}
\usepackage{caption}
\usepackage{subcaption}
\usepackage{placeins}
\usepackage[algo2e,ruled,vlined]{algorithm2e}
\usetikzlibrary {arrows.meta}
\usetikzlibrary {shapes.geometric,patterns,hobby}
\usetikzlibrary{spy}
\usetikzlibrary{shadings}
\usepackage{tikz-3dplot}
\usetikzlibrary{3d}
\usetikzlibrary{calc,arrows.meta,decorations.pathreplacing}
\usepackage{hyperref}
\usepackage{geometry}

\hypersetup{ colorlinks=true, linkcolor=black, filecolor=black, urlcolor=black }

\newtheorem{Theorem}{Theorem}[section]

\newtheorem{Lemma}[Theorem]{Lemma}

\newcommand{\vvv}{\vec{v}}

\newcommand{\vu}{\vec{u}}

\renewcommand{\vec}[1]{\mbox{\boldmath$#1$}}

\renewcommand{\vec}[1]{\mbox{\boldmath$#1$}}
\newcommand{\dx}{\,\mathrm{d}x}

\renewcommand{\vec}[1]{\mbox{\boldmath$#1$}}

\ifpdf
\DeclareGraphicsExtensions{.eps,.pdf,.png,.jpg}
\else
\DeclareGraphicsExtensions{.eps}
\fi

\title{Topology optimization for microfluidic mixers by a phase field method\thanks{This work was supported in part by the National Natural Science Foundation of China under grants (No. 12401534 and No. 12471377), the China Postdoctoral Science Foundation (No. 2024M751947), the Postdoctoral Fellowship Program (Grade B) of China Postdoctoral Science Foundation (No. GZB20240436), and the Science and Technology Commission of Shanghai Municipality (No. 22DZ2229014)}}
\author{Zongyuan Liu\thanks{School of Mathematical Sciences, East China Normal University, Shanghai 200241, China. E-mail: 51275500058@stu.ecnu.edu.cn
}
\and Jiajie Li \thanks{School of Mathematical Sciences, Shanghai Jiao Tong University, Shanghai 200240, China. E-mail: lijiajie20120233@163.com}
\and Shengfeng Zhu \thanks{Corresponding author. Key Laboratory of MEA (Ministry of Education) \& Shanghai Key Laboratory of Pure Mathematics and Mathematical Practice \& School of Mathematical Sciences, East China Normal University, Shanghai 200241, China. E-mail: sfzhu@math.ecnu.edu.cn}}

\begin{document}
\maketitle
\begin{abstract}
We investigate multi-physical topology optimization for microfluidic mixers employing the phase-field model. The optimization problem is formulated using a modified Ginzburg–Landau free energy functional. To eliminate fluid blockage in microfluidic mixers, we incorporate the coupled Navier–Stokes, convection–diffusion and Poisson–Boltzmann equations. An Allen–Cahn type gradient flow method is proposed based on sensitivity analysis. The algorithm is validated for its computational effectiveness through numerical simulations of benchmark problems in 2D and 3D. 
\end{abstract}


\section{Introduction}
Topology optimization of fluid flows has wide applications, including micromixer design \cite{Andreasen2009, DongYajiLiu2022,Ji2017,Xiong2023}, compressible flows \cite{Plotnikov2012Compressible}, and engine cooling \cite{Zhang2023}, as well as various multiphysics coupling scenarios \cite{Yaji2015ThermalFluid,Yaji2017RedoxFlowTO,Xia2023_TF_TO,Dede2014,Gallorini2023,LiYi2022}. Such problems aim to find a configuration or layout to optimize certain objective subject to geometric or physical constraints. Compared with shape optimization, which adjusts the geometry by controlling boundaries, topology optimization allows greater flexibility in generating and merging internal holes, thus enabling automatic material redistribution within the design domain \cite{BendsoeKikuchi1988,QianHuZhu2022,LiZhu2025}. 

The phase-field method\cite{Bendsoe2003,WZ,SokoJGA2023} has been widely used in topology optimization to effectively capture interface evolution and avoid complex explicit boundary tracking \cite{JinLiXuZhu2024,LiYangZhu2025}. Within this framework of using an auxiliary phase-field for modeling natural phase transitions, the material distribution is typically driven by a so-called Ginzburg-Landau free energy, whose gradient flow is often described by either the Allen–Cahn equation or the Cahn–Hilliard equation. When the temperature is below the critical point, the free energy typically exhibits a double-well potential, which represents the phenomenon of spontaneous symmetry breaking \cite{Landau1937}. Since then, the Ginzburg–Landau free energy has become a fundamental tool for studying symmetry breaking phenomena, as its potential structure captures the stability of different phases and the mechanisms of phase transitions such as the superconductivity \cite{deGennes1966}. In topology optimization problems, spontaneous symmetry breaking is also observed in fluid-related problems. Although the design domain and boundary conditions are often geometrically symmetric in configuration, the resulting optimal structures tend to exhibit asymmetrical features \cite{Andreasen2009,GersborgHansenSigmundHaber2005}. Therefore, to effectively track interface evolution \cite{LiXuZhu2025,HuQianZhu2023,CLWW2022} and enhance numerical stability, constructing phase-field models based on the Ginzburg–Landau-type free energy has become an efficient and robust strategy for fluid topology optimization. However, the phase-field method imposes certain requirements on smoothness and additional topological constraints. For many porous materials, however, complex microstructures cannot be efficiently captured. 

Among various fluid applications, microfluidic systems present unique challenges in flow control and mixing due to their small geometric scales and moderate Reynolds numbers. In such regimes, the flow remains laminar and the mixing efficiency is inherently low, making micromixing a key design priority. Therefore, enhancing mixing performance remains a key challenge in the development of high-performance microfluidic devices. Various strategies have been proposed to improve mixing in microfluidic systems, which are broadly classified into active and passive mixing approaches. Passive mixing strategies aim to enhance mixing by modifying the geometry of microchannels, such as through introducing obstacles, bends, or surface patterns to induce chaotic advection and increase interfacial contact between fluid streams~\cite{Kee2008,Hadjigeorgiou2021,Park2024,Du2010,Borrvall2003}. Passive methods are energy-efficient and robust, yet they are often limited by design constraints and flow conditions. Complementary to passive methods, active mixing involves the application of external fields, such as pressure fields, electric, magnetic, or thermal fields, to manipulate fluid flow and promote mixing \cite{Polson2000,Silva2024}. However, in the context of topology optimization for micromixing problems, inlet blockage frequently arises during optimization, especially in three-dimensional design domains. This phenomenon often results from the algorithm’s tendency to favor solid structures near the inlet to enhance mixing by redirecting flow, yet inadvertently impeding fluid entry. To address this issue, various strategies have been proposed, such as through applying spatial filters or incorporating pressure drop constraints into the objective~\cite{Dehghani2020,Guo2018}.

In this work, we derive the Ginzburg–Landau free energy functional by starting from the Ising model, and propose a modification that preserves its original physical meaning. This modification incorporates a higher-order interpolation function into improve regularity and enhance numerical stability. The proposed formulation is validated in both the single-physical field and the multiphysical field settings,  where the governing equations involve the coupling of the Navier–Stokes equations for fluid flow, the convection–diffusion equation for solute transport, and the Poisson–Boltzmann equation for electrokinetic effects. Furthermore, for the three-dimensional mixer problem, we propose a new strategy to mitigate inlet blockage, thus demonstrating the effectiveness of the method in complex multiphysics scenarios.

The rest of the paper is organized as follows. In Section \ref{Topology optimization in microfluidic mixer}, we first introduce the governing equations and the corresponding optimization problem. In Section \ref{Sensitivity analysis}, we derive the adjoint system of the optimization problem by using the Fréchet derivative and obtain the sensitivity of the phase-field variable. Based on this analysis, we propose a phase-field iteration scheme that preserves its underlying physical significance. In Section~\ref{Numerical examples}, we first present our algorithm and then test the proposed method on classical energy dissipation problems. Subsequently, we solve two and three-dimensional multiphysics-coupled mixer problems. Finally, we provide a brief summary in Section~\ref{conclusion}.

\section{Topology optimization in microfluidic mixer}\label{Topology optimization in microfluidic mixer}
We first introduce the coupled physical system associated with incompressible fluids, electric potential, and matter concentration. Then, we consider optimization problems with phase-field modeling.
\subsection{Governing equations}
Let $\Omega\subset\mathbb{R}^d(d=2,3)$ be an open, bounded domain with a Lipschitz boundary $\partial \Omega$. The boundary $\partial\Omega$ is partitioned into four disjoint components $\partial\Omega$ := $\Gamma_i\cup\Gamma_o\cup\Gamma_u\cup\Gamma_d$. Here, $\Gamma_i$ and $\Gamma_o$ represent the inlet and outlet, respectively. Additionally, $\Gamma_w:= \Gamma_d \cup \Gamma_u$ denotes the side walls of the channel. For the design domain with complex boundary configurations, the reader is referred to Fig. \ref{phaseFig} for more information.

\begin{figure}[htbp]
\centering
\centering
    \includegraphics[width=3.0 in]{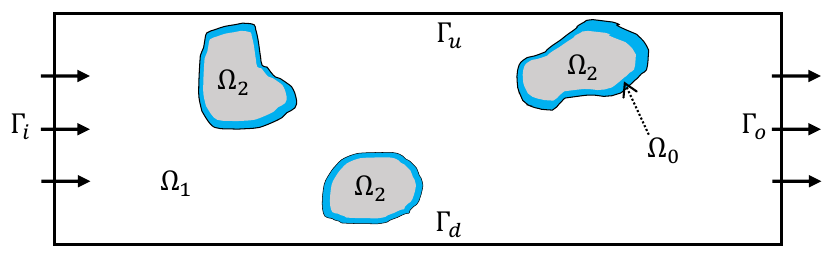}
 \caption{The design domain consisting of fluid $\Omega_1$ and solid $\Omega_1$ regions.}
 \label{phaseFig}
\end{figure}

The dimensionless steady-state incompressible Navier-Stokes equations are formulated as: Find the velocity field $\bm u: \Omega\rightarrow \mathbb{R}^d$ and pressure ${p}: \Omega\rightarrow \mathbb{R}$ such that
\begin{equation}\label{NS}
\left\{
\begin{aligned}
&-\frac{1}{\rm Re}\Delta\bm u  +(\bm u \cdot \nabla) \bm u +\alpha\bm u +\nabla {p} = \bm f,\ &&{\rm in}\ \Omega, \\
&\nabla\cdot \bm u  =0, &&{\rm in}\ \Omega, \\
&\bm u  =\bm u _0,\quad &&{\rm on}\ \Gamma_i, \\
&\bm u  = \bm 0,\ \ \quad &&{\rm on}\ \Gamma_w,\\
&(\nabla\bm u -{p}\mathbf{I}) \bm n = \bm 0,\quad &&{\rm on}\ \Gamma_o,
\end{aligned}\right.
\end{equation}
where Re $>0$ denotes the Reynolds number, \(\alpha: \Omega \to \mathbb{R}\) represents the inverse permeability, \(\bm{f}: \Omega \to \mathbb{R}^d\) denotes external body force, $\bm u_0: \Gamma_i\rightarrow \mathbb{R}^d$ is an inflow velocity on the inlet, $\mathbf{I}\in \mathbb{R}^{d\times d}$ denotes the identity matrix, and $\bm n$ is an outward unit vector normal to the boundary. 

Given a dielectric permittivity \(\varepsilon\), the dimensionless Poisson–Boltzmann equation for a symmetric monovalent electrolyte is formulated as: Find the electric potential \(\psi: \Omega \to \mathbb{R}\) such that
\begin{equation}\label{PB}
\left\{
    \begin{aligned}
        &-\nabla \cdot (\varepsilon^2 \nabla \psi)=\rho_e, \quad &&{\rm in}\ \Omega,\\
        &\rho_e=\rho_0 (e^{-\psi}-e^{\psi}), \quad &&{\rm in}\ \Omega,\\
        &\psi=\psi_i,\quad &&{\rm on}\ \Gamma_i, \\
        &\psi=\psi_w,\quad &&{\rm on}\ \Gamma_w, \\
        &\nabla \psi \cdot \bm n = 0,\quad &&{\rm on}\ \Gamma_o, \\
    \end{aligned}\right.
\end{equation}
where \(\rho_e: \Omega \to \mathbb{R}\) denotes the net charge density following the Boltzmann distribution with \(\rho_0\) being the bulk ionic concentration of both positive and negative charges. Here, $\psi_i:\Gamma_i\rightarrow\mathbb{R}$ and $\psi_w:\Gamma_w\rightarrow\mathbb{R}$ are the prescribed electric potentials on the inlet and wall, respectively. The electric field contributes to the external force term in \eqref{NS}, expressed as \(\bm{f} = -\rho_e \nabla \psi\), which influences the flow field dynamics.

The transport of concentration of substances driven by fluid flows is governed by the convection–diffusion equation: Find the concentration field $c: \Omega\rightarrow \mathbb{R}$ such that
\begin{equation}\label{cd}
\left\{
\begin{aligned}
&\bm u \cdot \nabla c - \frac{1}{\rm Pe} \Delta c =0,\quad &&{\rm in}\ \Omega, \\
&c = c_0 , \quad &&{\rm on}\ \Gamma_i, \\
&\nabla c \cdot \bm n = 0, \quad &&{\rm on}\ \Gamma_w \cup \Gamma_o,
\end{aligned}\right.
\end{equation}
where $\rm Pe$ denotes the Péclet number and $c_0$ is a prescribed inlet concentration.

Let us introduce some notations. Denote the Hilbert spaces $H^1(\Omega):=W^{1,2}(\Omega)$ and $H_0^1(\Omega):=\{v\in H^1(\Omega)\,|\, v= 0\ {\rm on}\ \partial\Omega\}$, where $W^{1,2}(\Omega)$ consists of square-integrable functions and their weak derivatives are also square-integrable. Introduce spaces of vectorial functions $\textbf{H}^1(\Omega):= H^1(\Omega)^d$ and $\textbf{H}^1_0(\Omega):= H^1_0(\Omega)^d$. Assuming $\bm u_0 \in \textbf{H}^{\frac{1}{2}}(\Gamma_i\cup \Gamma_w):=H^{\frac{1}{2}}(\Gamma_i\cup \Gamma_w)^d$, define $\textbf{H}^1_d (\Omega):=\{\vvv\in \textbf{H}^1(\Omega)\,|\ \vvv = \vu_0 \ {\rm on}\ \Gamma_i\cup\Gamma_w\}$, $\textbf{H}^1_{d0} (\Omega):=\{\vvv\in \textbf{H}^1(\Omega)\,|\, \ \vvv =\bm 0 \ {\rm on}\ \Gamma_i\cup\Gamma_w\}$ and $L_0^2(\Omega):=\{q\in L^2(\Omega)\vert \int_\Omega q\dx = 0 \}$. Given the boundary data $\psi_i\in L^2(\Gamma_i)$ and $\psi_w\in L^2(\Gamma_w)$, we denote the set $H_\psi^1(\Omega):=\{ \xi\in H^1(\Omega)\,|\, \xi=\psi_i\ {\rm on}\ \Gamma_i, \ \xi=\psi_w  \ {\rm on}\ \Gamma_w\}$ for the electric potential and $H_{\psi0}^1(\Omega):=\{ \xi\in H^1(\Omega)\,|\, \xi=0\ {\rm on}\ \Gamma_i, \ \xi=0  \ {\rm on}\ \Gamma_w\}$. Given $c_0\in L^2(\Gamma_i)$, introduce a set for the concentration as $H_c^1(\Omega):=\{ s\in H^1(\Omega)\,|\, s=c_0\ {\rm on}\ \Gamma_i  \}$, and a space $H_{c0}^1(\Omega):=\{ s\in H^1(\Omega)\,|\, s=0\ {\rm on}\ \Gamma_i  \}$.
\subsection{Phase-field models} \label{Phase-field models}
In topology optimization, the design problem is typically formulated as the minimization of an objective functional subject to the PDEs and geometric constraints. Alternatively, the optimization process can be interpreted through the dynamics of a phase-field model. The domain $\Omega$ is decomposed into three disjoint subdomains $\Omega_1$, $\Omega_2$, and $\Omega_0$, such that $\overline{\Omega}=\overline{\Omega}_1\cup \overline{\Omega}_2 \cup \overline{\Omega}_0$, where $\Omega_1$, $\Omega_2$ and $\Omega_0$ denote the fluid region, solid region, and porous medium, respectively (see Fig. \ref{phaseFig}). The subdomains of $\Omega$ are characterized implicitly by a phase-field function such that
\begin{equation*}\left\{
\begin{aligned}
\phi(\bm x)=1&, \quad &&\bm x\in \Omega_1,\\
0<\phi(\bm x)< 1&,&&\bm x\in  \Omega_0,\\
\phi(\bm x)=0&, \quad &&\bm x\in  \Omega_2.
\end{aligned}\right.
\end{equation*}
Introduce an admissible set $\mathcal{O}_{\rm ad}:=\{\phi\in H^1(\Omega)\cap L^\infty(\Omega)|\  0\leq\phi\leq1\ {\rm a.e.\ in\ \Omega} \}$ for the phase-field function. For a given $\phi\in \mathcal{O}_{\rm ad}$, the inverse permeability is defined as $\alpha(\phi):=\alpha_0(1-\phi)$ with $\alpha_0>0$ being a prescribed coefficient, and the dielectric permittivity is $\varepsilon(\phi):=\varepsilon_0(1-\phi)+\varepsilon_m\phi,\ 0<\varepsilon_m \ll 1$. We consider the weak formulation of the Navier-Stokes equations coupled with Poisson-Boltzmann and convection–diffusion equations: Given $\phi\in \mathcal{O}_{\rm ad}$, find $(\bm u, p, \psi, c)\in \textbf{H}^1_d(\Omega) \times L^2_0(\Omega)\times H_\psi^1(\Omega)\times  H^1_c(\Omega)$ such that
\begin{equation}\label{weakNSPBC}
\small
\left\{
  \begin{aligned}  
&\int_\Omega \varepsilon(\phi)\nabla\psi\cdot\nabla\xi-\rho_0(e^{-\psi}-e^{\psi})\xi\dx=0,&&\forall \xi\in \textbf{H}_{\psi0}^1(\Omega),\\
&\int_\Omega \frac{1}{\rm Re}\nabla \bm u:\nabla\bm v+ (\bm u\cdot \nabla)\bm u\cdot \bm v- p \nabla\cdot \bm v+\alpha(\phi)\bm u\cdot \vvv+\rho_0 (e^{-\psi}-e^{\psi})\nabla\psi\cdot \vvv\dx=0, &&\forall\vvv\in \textbf{H}_{d0}^1(\Omega),\\
&\int_\Omega q \nabla\cdot \bm u \dx =0,&& \forall q\in L^2(\Omega),\\
&\int_\Omega (\bm u \cdot \nabla c)s + \frac{1}{\rm Pe}\nabla c\cdot \nabla s\dx = 0, &&\forall s\in H_{c0}^1(\Omega).
\end{aligned}\right.
\end{equation}


Denote by $$ F(\phi):=\int_\Omega{\bigg[\frac{\kappa}{2} |\nabla \phi|^2 + w(\phi)\bigg]\dx}$$ the Ginzburg-Landau free energy with a diffusive constant $\kappa>0$ and a double-well potential
\begin{equation}\label{dp}
w(\phi) = \frac{1}{4} \phi^2 (1 - \phi)^2.
\end{equation}
We consider four types of topology optimization problems:
\begin{itemize}
    \item \textbf{Energy dissipation problem:} We consider the classic energy dissipation problem, which has applications, e.g., design of biological fluids and microfluidic devices. Consider to minimize the energy dissipation 
\begin{equation}\label{J1}
    \min_{\phi \in \mathcal{O}_{\rm ad}} J_1(\phi):=\int_{\Omega}\frac{1}{2\rm Re}\vert\nabla\bm u(\phi) \vert^2+\frac{\alpha(\phi)}{2}\vert\bm u(\phi) \vert^2\dx,
\end{equation}
to reduce the interfacial resistance within liquids, thereby enhancing the overall fluid transport efficiency, where $\bm{u}=\bm{u}(\phi)$ is the velocity field of \eqref{weakNSPBC}.
The first term in \eqref{J1} denotes the viscous dissipation energy, while the second term in \eqref{J1} represents the penalty imposed on the fluid within the solid region. The total free energy of phase-field modeling is formulated as 
\begin{equation}\label{W1}
    \mathcal{W}_1(\phi) := F(\phi) +\int_\Omega\eta\phi^3(6\phi^2-15\phi+10)\frac{J_1'(\phi)}{\|J_1'(\phi)\|}\dx+\frac{\beta_3}{2} (V(\phi) - V_0)^2,
\end{equation}
where $\eta,\beta_3$ are positive constants, and a volume constraint of the solid phase is incorporated by penalty with
\begin{equation*}
     \quad V(\phi):=\int_\Omega{(1-\phi)}\dx
\end{equation*}
and $V_0>0$ being the target volume of the solid domain. 

    \item \textbf{Mixer problem (single objective):} The main topology optimization problem we consider is the mixer design problem \cite{Polson2000,Silva2024}, where a mixing agent is introduced through the inlet tube. The governing equations are the Navier-Stokes equations (\ref{NS}) and convection-diffusion equation (\ref{cd}). To achieve more uniform mixing at the outlet, consider
     \begin{equation}\label{J2}
     \min_{\phi \in \mathcal{O}_{\rm ad}}J_2(\phi):=\frac{1}{2}\int_{\Gamma_o}(c(\phi)-c_d)^2{\rm d}\Gamma,
 \end{equation}
subject to the solid-phase volume constraint $V(\phi)=V_0$. Here, $c_d\in L^2(\Gamma_o)$ denotes the target average concentration at the outlet. Consider optimizing the total free energy as 
\begin{equation}\label{W2}
    \mathcal{W }_2(\phi):=F(\phi)+ \int_\Omega\eta\phi^3(6\phi^2-15\phi+10)\frac{J_2'(\phi)}{\|J_2'(\phi)\|}\dx+\frac{\beta_3}{2} (V(\phi) - V_0)^2.
\end{equation}

    \item \textbf{Electrokinetically enhanced mixer problem:} The governing system includes the Poisson-Boltzmann equation (\ref{PB}), Navier-Stokes equations (\ref{NS}), and convection-diffusion equation (\ref{cd}). The objective is to minimize $J_2$. The associated free energy takes the same form as $\mathcal{W}_2$.

    \item \textbf{Mixer problem (multi-objective):} The governing equations are the Navier-Stokes equations coupled
with Poisson-Boltzmann and convection–diffusion equations. The objective is a weighted combination of energy dissipation and mixing performance: $\beta_1 J_1 + \beta_2 J_2$. The corresponding model problem reads as
\begin{equation}\label{W3}
    \mathcal{W}_3(\phi):=F(\phi)+\int_\Omega\eta\phi^3(6\phi^2-15\phi+10)\frac{\beta_1J_1'(\phi)+ \beta_2J_2'(\phi)}{\|\beta_1J_1'(\phi)+ \beta_2J_2'(\phi)\|}\dx+\frac{\beta_3}{2} (V(\phi) - V_0)^2,
\end{equation}
where $\beta_1, \beta_2$ is a positive constant and the modified forms of free energy will be explained in Section \ref{Gradient flow}.
\end{itemize}
\section{Phase field method}\label{Sensitivity analysis}
We deduce the weak form of the adjoint problem corresponding to the optimal control type of model problem. Then the sensitivity analysis is performed via obtaining the Fr\'echet derivative of the objective. A gradient flow method of the Allen-Cahn type is proposed for the model problem.
\subsection{Sensitivity analysis}
\begin{Lemma}
Let $\phi\in \mathcal{O}_{\rm ad}$ and $(\bm u, p, \psi, c)\in \textbf{H}^1_d(\Omega) \times L^2_0(\Omega)\times H^1_\psi(\Omega) \times H^1_c(\Omega)$ be the solution of the coupled Navier-Stokes equations \eqref{weakNSPBC}. Then, the weak formulation of the adjoint system is given by: Find $(\vvv,q,\xi,s)\in \textbf{H}_{d0}^1(\Omega)\times L^2_0(\Omega)\times H_{\psi0}^1(\Omega)\times H_{c0}^1(\Omega)$ such that
\begin{equation}\label{weakAdj}
\left\{
    \begin{aligned}
        &\int_{\Omega}\frac{1}{\rm Re}\nabla \vvv:\nabla\bm y+(\bm y \cdot\nabla)\bm u \cdot\vvv-(\bm u \cdot\nabla)\bm \vvv \cdot {\bm y}+\alpha(\phi) \vvv\cdot\bm y+q\nabla\cdot\bm y +(\bm y \cdot\nabla c)s\dx\\
        &\quad=\beta_1 \int_\Omega\frac{1}{\rm Re}\nabla\vu: \nabla \bm y+\alpha(\phi)\vu\cdot\bm y\dx,\ \forall\bm y\in \textbf{H}_{d0}^1(\Omega),\\
        &\int_\Omega r\nabla\cdot\vvv\dx=0,\ \forall r\in L^2(\Omega),\\
        &\int_\Omega \varepsilon(\phi)\nabla\xi\cdot\nabla \chi+\rho_0(e^{-\psi}+e^{\psi})\xi\ \chi-\rho_0(e^{-\psi}+e^\psi)\nabla\psi\cdot\vvv\ \\
        &\quad +\rho_0(e^{-\psi}-e^\psi)\nabla \chi\cdot\vvv\dx=0,\ \forall \chi\in H^1_{\psi0}(\Omega)\\
        &\int_\Omega (\bm u \cdot\nabla z)s+\frac{1}{\rm Pe}\nabla s\cdot\nabla z\dx=\int_{\Gamma_o}{\beta_2(c-c_d)z {\rm d}\Gamma},\ \forall z\in H^1_{c0}(\Omega).\\
    \end{aligned} \right.
\end{equation}
\end{Lemma}
\begin{proof}
By utilizing the weak formulation of the coupled system (\ref{weakNSPBC}) and the objectives (\eqref{J1} and \eqref{J2}), we define
\begin{equation}\label{mixingL}
\begin{aligned}
    &\mathcal{L}(\phi, \psi, \bm u, p, c; \zeta, \bm v, q, s)\\
    :=&\int_{\Gamma_0}{\frac{\beta_2}{2}(c-c_d)^2}{\rm d}\Gamma+\frac{\beta_1}{2}\int_{\Omega}\frac{1}{\rm Re}\vert\nabla\bm u \vert^2+\alpha(\phi)\vert\bm u \vert^2\dx\\ 
    &-\int_{\Omega}\frac{1}{\rm Re}\nabla\bm u :\nabla\vvv+(\bm u \cdot\nabla)\bm u \cdot\vvv-p\nabla\cdot\vvv+\alpha(\phi)\bm u \cdot\vvv+\rho_0 (e^{-\psi}-e^{\psi})\nabla\psi\cdot\vvv\dx\\
    &-\int_\Omega \varepsilon(\phi)\nabla\psi\cdot\nabla\xi-\rho_0(e^{-\psi}-e^{\psi})\xi\dx\\
    &-\int_\Omega{q\nabla \cdot\bm u }\dx-\int_{\Omega}(\bm u \cdot\nabla c)s+\frac{1}{\rm Pe}\nabla c\cdot\nabla s\dx,
\end{aligned}
\end{equation}
where $(\vu,p,\psi,c)$ is the solution of \eqref{weakNSPBC} with the prescribed phase-field function and $(\vvv,q,\xi,s)\in \textbf{H}_{d0}^1(\Omega)\times L^2(\Omega)\times H_{\psi0}^1(\Omega)\times H_{c0}^1(\Omega)$. Let the Fr\'{e}chet derivative of $\mathcal{L}$ with respect to each state variable be zero
\begin{equation}\label{mixingAdjointSystem}
\left\{
    \begin{aligned}
        & \left\langle\frac{\partial \mathcal{L}}{\partial \bm u},\delta_{\vu} \right\rangle=0,\ \forall \delta_{\vu} \in \textbf{H}_{d0}^1(\Omega), 
        &&\left\langle\frac{\partial \mathcal{L}}{\partial p}, \delta_p \right\rangle=0, \ \forall\delta_p\in L^2(\Omega), \\
        &\left\langle\frac{\partial \mathcal{L}}{\partial \psi}, \delta_\psi\right\rangle=0,\ \forall\delta_\psi \in H_{\psi0}^1(\Omega),
        &&\left\langle\frac{\partial \mathcal{L}}{\partial c}, \delta_c\right\rangle=0,\ \forall\delta_c \in H_{c0}^1(\Omega).\\
    \end{aligned}\right.
\end{equation}
The expanded form of system \eqref{mixingAdjointSystem} is given as follows:
\begin{equation*}
\left\{
    \begin{aligned}
        &\int_{\Omega}\frac{1}{\rm Re}\nabla\delta_{\bm u} :\nabla \vvv+(\delta_{\bm u} \cdot\nabla)\bm u \cdot\vvv+(\bm u \cdot\nabla)\delta_{\bm u} \cdot\vvv+\alpha(\phi)\delta_{\bm u} \cdot\vvv+q\nabla\cdot\delta_{\bm u} +(\delta_{\bm u} \cdot\nabla c)s\dx\\ 
        &\quad=\beta_1 \int_\Omega\frac{1}{\rm Re}\nabla\vu: \nabla \delta_{\bm u}+\alpha(\phi)\vu\cdot\delta_{\bm u}\dx,\\ 
        &\int_\Omega\delta_p\nabla\cdot\vvv\dx=0,\\ 
        &\int_\Omega \varepsilon(\phi)\nabla\xi\cdot\nabla\delta_\psi+\rho_0(e^{-\psi}+e^{\psi})\xi\ \delta_\psi-\rho_0(e^{-\psi}+e^\psi)\nabla\psi\cdot\vvv\ \delta_\psi+\rho_0(e^{-\psi}-e^\psi)\nabla\delta_\psi\cdot\vvv\dx=0,\\ 
        &\int_\Omega (\bm u \cdot\nabla \delta_c)s+\frac{1}{\rm Pe}\nabla \delta_c\cdot\nabla s\dx=\int_{\Gamma_o}{\beta_2(c-c_d)\delta_c {\rm d}\Gamma}.\\
    \end{aligned}\right.
\end{equation*}
By the divergence free property $\nabla\cdot \vu =0$, then it holds that
$\int_\Omega (\bm u \cdot \nabla)\delta_{\bm u}\cdot \vvv \dx = -\int_\Omega (\vu \cdot \nabla)\vvv \cdot \delta_{\bm u}  \dx$ (see \cite[Lemma 2.4]{GarckeHecht2016}). Then, we have
\begin{equation*}
\left\{
    \begin{aligned}
        &\int_{\Omega}\frac{1}{\rm Re}\nabla\delta_{\bm u} :\nabla \vvv+(\delta_{\bm u} \cdot\nabla)\bm u \cdot\vvv-(\vu \cdot \nabla)\vvv \cdot \delta_{\bm u}+\alpha(\phi)\delta_{\bm u}\cdot\vvv+q\nabla\cdot\delta_{\bm u} +(\delta_{\bm u} \cdot\nabla c)s\dx\\ 
        &\quad=\beta_1 \int_\Omega\frac{1}{\rm Re}\nabla\vu: \nabla \delta_{\bm u}+\alpha(\phi)\vu\cdot\delta_{\bm u}\dx,\\ 
        &\int_\Omega\delta_p\nabla\cdot\vvv\dx=0,\\ 
        &\int_\Omega \varepsilon(\phi)\nabla\xi\cdot\nabla\delta_\psi+\rho_0(e^{-\psi}+e^{\psi})\xi\ \delta_\psi-\rho_0(e^{-\psi}+e^\psi)\nabla\psi\cdot\vvv\ \delta_\psi+\rho_0(e^{-\psi}-e^\psi)\nabla\delta_\psi\cdot\vvv\dx=0,\\ 
        &\int_\Omega (\bm u \cdot\nabla \delta_c)s+\frac{1}{\rm Pe}\nabla \delta_c\cdot\nabla s\dx=\int_{\Gamma_o}{\beta_2(c-c_d)\delta_c {\rm d}\Gamma}.\\
    \end{aligned}\right.
\end{equation*}
Hence, we complete the derivation of adjoint problem in weak form.
\end{proof}

Taking the Fr\'{e}chet derivative of the Lagrangian $\mathcal{L}$ defined in \eqref{mixingL} with respect to $\phi$ in the direction $\theta\in H^1(\Omega)$, we have
\begin{equation}\label{varDelStep1}
\begin{aligned}
&\left\langle\beta_1 J_1^\prime(\phi)+\beta_2 J_2^\prime(\phi),\theta\right\rangle\\
=&\left\langle \frac{\partial\mathcal{L}}{\partial\phi}(\phi, \psi, \bm u, p, c; \zeta, \bm v, q, s),\theta \right\rangle+\left\langle \frac{\partial\mathcal{L}}{\partial c}(\phi, \psi, \bm u, p, c; \zeta, \bm v, q, s),\langle c^\prime(\phi),\theta \rangle \right\rangle \\
&+\left\langle \frac{\partial\mathcal{L}}{\partial\bm u}(\phi, \psi, \bm u, p, c; \zeta, \bm v, q, s),\langle\bm u^\prime(\phi),\theta \rangle \right\rangle+\left\langle \frac{\partial\mathcal{L}}{\partial p}(\phi, \psi, \bm u, p, c; \zeta, \bm v, q, s),\langle p^\prime(\phi),\theta \rangle \right\rangle\\
&+\left\langle \frac{\partial\mathcal{L}}{\partial \psi}(\phi, \psi, \bm u, p, c; \zeta, \bm v, q, s),\langle \psi^\prime(\phi),\theta \rangle \right\rangle.
\end{aligned}
\end{equation}
Using the weak formulation of the adjoint problem \eqref{weakAdj}, the last four terms of the right hand side of \eqref{varDelStep1} vanish. Therefore, we derive the variational derivative
\begin{equation}
\left\langle\beta_1 J_1^\prime(\phi)+\beta_2 J_2^\prime(\phi),\theta\right\rangle= \frac{\beta_1}{2} \int_\Omega \frac{1}{\rm Re}\alpha^\prime(\phi)\theta\vert \bm u\vert^2\dx- \int_\Omega  \alpha^\prime(\phi)\theta \bm u\cdot\bm v\dx- \int_\Omega \varepsilon^\prime(\phi)\theta\nabla\psi\cdot\nabla\xi\dx.
\end{equation}

\subsection{Gradient flow method}\label{Gradient flow}
In thermodynamics, the evolution of a system generally proceeds in the direction of decreasing free energy. Depending on the specific process and the type of particle system, free energy can be expressed in various forms, among which the Ginzburg–Landau free energy is the most commonly employed to describe phase transitions. Similarly, in the field of topology optimization, this formulation is often adopted as the iterative scheme for phase-field evolution. Motivated by thermodynamic principles, we regard topology optimization via the phase field method as analogous to a phase transition near the critical temperature, where the sensitivity of the objective function naturally plays the role of an external field. In this view, the optimization process can be interpreted as the system evolving under the combined driving forces of intrinsic free energy and the external field, ultimately reaching a stable equilibrium configuration.

To describe continuous phase transitions in homogeneous media, the Landau–Ginzburg mean-field theory is commonly employed due to its analytical tractability and generality. In this framework, a smooth phase-field variable $\phi$ is introduced to represent the material distribution throughout the design domain (see Fig. \ref{phaseFig}). The evolution of $\phi$ is then driven by the minimization of a Ginzburg–Landau-type free energy functional, and the optimization proceeds by solving the associated Allen–Cahn equation:
\begin{equation}\label{AllenCahn}
\frac{\partial \phi}{\partial t} = -\frac{\delta \mathcal{F}(\phi)}{\delta \phi},
\end{equation}
where $\mathcal{F}(\phi)$ denotes the functional total free energy density and the initial condition and Neumann boundary condition are satisfied.
According to \cite{Goldenfeld1992}, the Ginzburg–Landau free energy originates from a coarse-grained continuum approximation of the Ising model. For completeness, we provide the detailed derivation in Appendix \ref{appendix:IsingGL}. The functional free energy density can be expressed as 
$$\mathcal{F}(\phi)=\frac{\kappa}{2}|\nabla\phi|^2+A(T)\phi^2+B(T)\phi^4+o(\phi^4).$$
Here
$$\kappa:=\frac{a^2k_BT}{2},\quad A(T):=\frac{zk_BT}{2}-\frac{z^2\tilde{J}}{2},\ {\rm and} \ B(T):=\frac{z^4\tilde{J}^2}{12k_BT},$$
with $T$ represents the temperature, $k_B$ is the Boltzmann constant, $a$ is the lattice spacing constant, $z$ is the coordination number, and $\tilde{J}$ is a positive constant representing the nearest-neighbor coupling strength.  

it is observed that $B(T)>0$ ensures the stability of the model, and the coefficient $A(T)$ changes sign at a critical temperature $T_c:=\frac{z\tilde{J}}{k_B}$. Specifically, for $T>T_c$, one has $A(T)>0$, and the local free energy landscape corresponds to a single-well potential centered at $\phi=0$. However, when $T<T_c$,  $A(T)<0$, the free energy undergoes a qualitative change to a double-well structure. This transition is referred to as spontaneous symmetry breaking: although the microscopic formulation assumes symmetry among spin states, the system spontaneously selects one of the two degenerate minima, thereby breaking the original symmetry.


From this viewpoint, we re-derive the Ginzburg–Landau free energy expression in the presence of an external field from the ising model.  
By performing a Taylor expansion around $\phi=0$, a modified Ginzburg–Landau free energy functional can be obtained:
\begin{equation*}
\begin{aligned}
\mathcal{F}(\phi) =& \frac{zk_BT}{2}\phi^2+\frac{a^2k_BT}{4}|\nabla\phi|^2-k_BT\ln[2\cosh(z\sqrt{\frac{\tilde{J}}{k_BT}}\phi+\frac{h}{k_BT})]\\
=&\frac{a^2k_BT}{4}|\nabla\phi|^2-k_BT\ln(2\cosh (\frac{h}{k_BT}))-z\sqrt{\tilde{J}k_BT}\tanh(\frac{h}{k_BT})\phi\\
&+(-\frac{z^2\tilde{J}}{2\cosh^2(\frac{h}{k_BT})}+\frac{zk_BT}{2})\phi^2+z^3\sqrt{\frac{\tilde{J}^3}{k_BT}}\frac{\sinh(\frac{h}{k_BT})}{3\cosh^3(\frac{h}{k_BT})}\phi^3\\
&+\frac{z^4\tilde{J}^2\bigl(1 - 3\tanh^2 (\frac{h}{k_BT})) }{12k_BT\cosh^2(\frac{h}{k_BT})}\phi^4+o(\phi^4),\\
\end{aligned}
\end{equation*}
Here, $h$ denotes the external field. The expression implies that, when an external field is present, the original Ginzburg–Landau free energy expression is modified by the appearance of odd-order terms, together with slight perturbations to the coefficients of the even-order terms.

Therefore, in the standard Ginzburg–Landau free energy, each point of the phase field is driven to select between the two states. When an external field is introduced, odd-order terms arise, which break the symmetry of the double-well potential and render the free energy values at the two minima unequal. Although the system still tends to approach these two critical points, the distribution is altered due to the influence of the odd terms. However, in topology optimization, it is essential to ensure that the distribution function always approaches pure phases $0$ and $1$. Moreover, $\tilde{J}$ represents the free energy difference between distinct states of the phase field; in topology optimization, the state parameter is directly characterized by $\phi$, which means that the magnitude of the odd-order terms is governed by the derivative of the objective function $J’$. To this end, we follow the approach of \cite{Takezawa2010} and adopt the modified double-well potential \eqref{dp}, which shifts the energy minima toward the desired values.
For the external field perturbation term, we further introduce a higher-order interpolation function 
\begin{equation*}
    g(\phi)=\eta\phi^3(6\phi^2-15\phi+10)\frac{J'}{\|J'\|},
\end{equation*}
allowing the objective to act as an external driving force for the evolution of the phase-field function, where the constant $\eta>0$ serves as a tuning parameter to control the strength of this external effect. The modified Ginzburg-Landau free energy density functional is
\begin{equation}
\begin{aligned}
    \mathcal{F}(\phi)&=\frac{\kappa}{2}|\nabla\phi|^2+w(\phi)+g(\phi)\\
    &=\frac{\kappa}{2}|\nabla\phi|^2+\frac{1}{4} \phi^2 (1 - \phi)^2+\eta\phi^3(6\phi^2-15\phi+10)\frac{J'}{\|J'\|},\\
\end{aligned}
\end{equation}
where in each time step the sensitivity analysis indicates that $J'$ does not explicitly involve the state variable $\phi$, so the derivative of this term with respect to $\phi$ is zero.
The influence of the objective on the distribution of phase-field function is illustrated in Fig. \ref{GLFreeEnergy}.
\begin{figure}[htbp]
    \centering
    \includegraphics[width=0.9\linewidth]{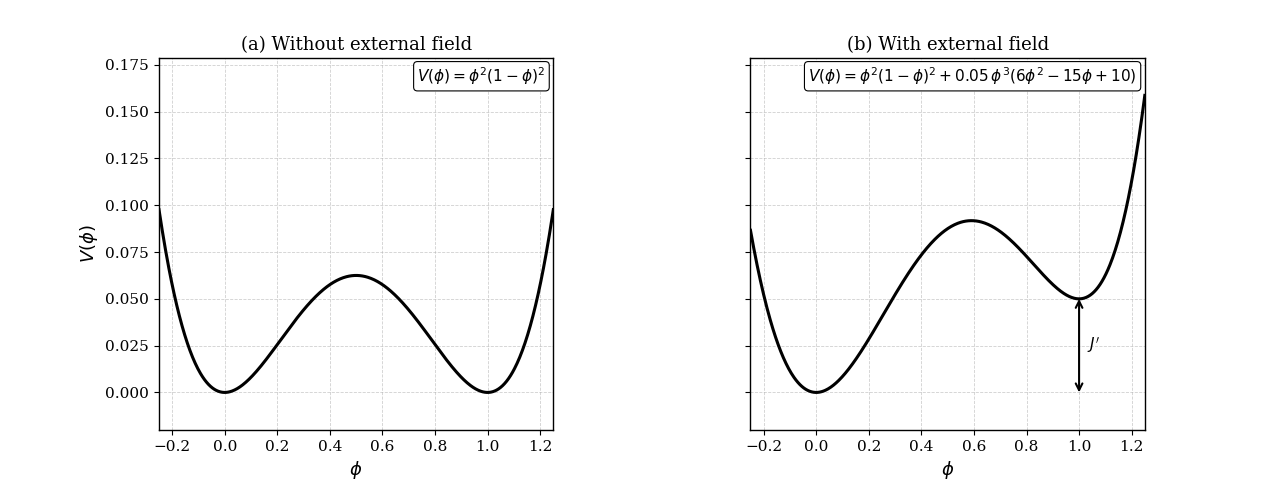}
    \caption{Ginzburg-Landau free energy under different conditions}
    \label{GLFreeEnergy}
\end{figure}

Applying the Allen-Cahn equation (\ref{AllenCahn}), for the general form of the problems discussed in Section \ref{Phase-field models}, we have
\begin{equation*}
\begin{aligned}
    \frac{\partial \phi}{\partial t} 
    =\kappa\Delta\phi+\phi(1-\phi)r(\phi)+\beta_3(V(\phi)-V_0),
\end{aligned}
\end{equation*}
where $$
r(\phi)=\phi-\frac{1}{2}-30\eta(1-\phi)\phi\frac{\beta_1 J_1' + \beta_2 J_2'}{\|\beta_1 J_1' + \beta_2 J_2'\|}.
$$
In the numerical implementation, let $\phi^n$ denote the phase distribution at $t^n=n\tau$, $\tau$ be the time step.
We adopt the following semi-implicit scheme:
\begin{equation}\label{ModifiedAC}
    \begin{aligned}
        \frac{\phi^{n+1}-\phi^n}{\tau}=\kappa\Delta\phi^{n+1}+\beta_3(V(\phi^n)-V_0)+\begin{cases}
\phi^{n+1} \left(1 - \phi^n\right) r\left(\phi^n\right), & \text{if } r(\phi^n) \leq 0, \\
\phi^n \left(1 - \phi^{n+1}\right) r\left(\phi^n\right), & \text{if } r(\phi^n) > 0,
\end{cases}
    \end{aligned}
\end{equation}
where $J_1^\prime$ and $J_2^\prime$ are treated explicitly. According to \cite{Takezawa2010}, the above discretization guarantees that $\phi$ remains in the interval $[0,1]$ for a large time step.

\section{Numerical simulation}   \label{Numerical examples}

In this section, we first verify our scheme with an energy‑dissipation benchmark. Two-dimensional and three-dimensional numerical examples are tested. After the state and adjoint variables are obtained, the Allen–Cahn phase-field equation is advanced for $N_p$ steps according to \eqref{ModifiedAC}. We then apply the scheme to a mixer problem in which two fluids enter through the inlet and should be mixed as uniformly as possible on the outlet. The corresponding implementation procedures for the various coupled‑physics settings are listed in Algorithm \ref{algPB}. For both problems, we adopt a fixed total iteration number $N$ as the stopping criterion to show that the proposed algorithm converges monotonically without oscillation over long runs. The steady incompressible Navier–Stokes equations are solved using Newton's method. For the case of the coupled electrical field, we also use Newton's method to solve the Poisson-Boltzmann equation. 

For finite element discretization, continuous piecewise linear elements enriched with bubble functions ($P_1$-bubble element) are employed for the velocity field and its adjoint variable, while standard linear elements ($P_1$ element) are used for the pressure, concentration, electric potential, phase-field variable, and their corresponding adjoint quantities. Numerical experiments are performed in FreeFem++ on a computer equipped with an Apple M4 processor and 32 GB of RAM.

\begin{algorithm}[h!]
\SetAlgoLined
    \caption{Topology optimization for mixer problem}\label{algPB}
    \KwData{Given the maximum iteration number $N$, iteration time $n=0$\\
    Initialize phase-field function $\phi_0$}
    \While{$n \leq N$}{
     \, \emph{Step 1}: Solve the Poisson-Boltzmann equation;\\ 
    \ \emph{Step 2}:  Solve the Navier-Stokes equations; \\
    \ \emph{Step 3}: Solve the convection-diffusion equation;\\
    \ \emph{Step 4}: Solve adjoint system;\\
    \ \emph{Step 5}: Compute the gradient of the objective; \\
\ \emph{Step 6}: Update the phase-field function via Allen-Cahn gradient flow;\\
     $n \leftarrow n+1$}
\end{algorithm}
\FloatBarrier

\subsection{Energy dissipation problem}
We first test Algorithm \ref{algPB} by solving energy dissipation problems constrained by the Navier-Stokes equation \eqref{NS} in 2D and 3D. Set $\beta_1=1, \beta_2=0, N_p=3$.

$\rm \bf Example\ 1:$ The first example is the shape design for the bypass which has application in the aorto-coronaric bypass anastomoses. Set $\Omega=[0,1.8]\times[-0.5,0.5]$ (see Fig. \ref{energy dissipation domain for example 1}) with two inlets and two outlets, and a no-slip condition is imposed on the other boundary. The prescribed inflow velocity is imposed as ${\bm u_0}=[20(0.35-y)(0.35+y),0]^{\rm T}$. Set the target volume to be $V_0=1.25$ and choose the initial phase-field function $\phi_0(x,y)=\min\{|y-0.35|-0.15,|y+0.35|-0.15\}$. With the parameters: ${\rm Re=13.69}, \alpha=10^2,\kappa=5\times 10^{-3}, {\tau}=10^{-3}, \eta=75, \beta_3=3.5\times 10^2, N=40$, the optimized distributions (red: fluid and blue: solid) with the velocity profiles and the pressure contours are exhibited in Fig. \ref{Optimization distribution with the correspond-
ing velocity profiles for example 1} and Fig. \ref{Pressure distribution with the corresponding velocity profiles for example 1}.

\begin{figure}[htbp]
\centering

\subfigure[Design domain]{
    \begin{minipage}[t]{0.45\linewidth}
    \centering
    \begin{tikzpicture}[yscale=0.7,xscale=1.4]
    \draw (0,0) rectangle (3,3);
    \draw[arrows = {-Stealth[width=3pt, length=3pt]}](0,2.475)--(0.16,2.475);
    \draw[arrows = {-Stealth[width=3pt, length=3pt]}](0,2.35)--(0.16,2.35);
    \draw[arrows = {-Stealth[width=3pt, length=3pt]}](0,2.225)--(0.16,2.225);
    \draw[arrows = {-Stealth[width=3pt, length=3pt]}](0,0.525)--(0.16,0.525);
    \draw[arrows = {-Stealth[width=3pt, length=3pt]}](0,0.65)--(0.16,0.65);
    \draw[arrows = {-Stealth[width=3pt, length=3pt]}](0,0.775)--(0.16,0.775);
    
    \draw[arrows = {-Stealth[width=3pt, length=3pt]}](3,2.7)--(3.2,2.7);
    \draw[arrows = {-Stealth[width=3pt, length=3pt]}](3,2.46)--(3.2,2.46);
    \draw[arrows = {-Stealth[width=3pt, length=3pt]}](3,2.23)--(3.2,2.23);
    \draw[arrows = {-Stealth[width=3pt, length=3pt]}](3,2.0)--(3.2,2.0);
    \draw[arrows = {-Stealth[width=3pt, length=3pt]}](3,0.3)--(3.2,0.3);
    \draw[arrows = {-Stealth[width=3pt, length=3pt]}](3,0.53)--(3.2,0.53);
    \draw[arrows = {-Stealth[width=3pt, length=3pt]}](3,0.76)--(3.2,0.76);
    \draw[arrows = {-Stealth[width=3pt, length=3pt]}](3,1)--(3.2,1);
    
    \draw[<->={-Stealth[width=3pt, length=3pt]}](-0.05,1.05)--(-0.05,1.95);
    \draw (-0.2,1.5) node[font=\small, scale=0.8]{$0.4$};
    \draw[<->={-Stealth[width=3pt, length=3pt]}](-0.05,2)--(-0.05,2.6);
    \draw (-0.25,2.3) node[font=\small, scale=0.8]{$0.15$};
    \draw[<->={-Stealth[width=3pt, length=3pt]}](-0.05,0.95)--(-0.05,0.35);
    \draw (-0.25,0.65) node[font=\small, scale=0.8]{$0.15$};
    
    \draw[<->={-Stealth[width=3pt, length=3pt]}](0,-0.2)--(3,-0.2);
    \draw (1.5,-0.5) node[font=\small, scale=0.8]{$1.8$};
    
    \draw[<->={-Stealth[width=3pt, length=3pt]}](2.95,2.0)--(2.95,2.7);
    \draw (2.75,2.35) node[font=\small, scale=0.8]{$0.2$};
    \draw[<->={-Stealth[width=3pt, length=3pt]}](2.95,0.3)--(2.95,1.0);
    \draw (2.75,0.6) node[font=\small, scale=0.8]{$0.2$};
    \draw[<->={-Stealth[width=3pt, length=3pt]}](2.95,1.95)--(2.95,1.05);
    \draw (2.75,1.5) node[font=\small, scale=0.8]{$0.4$};

    \end{tikzpicture}
     \label{energy dissipation domain for example 1}
 
    \end{minipage}%
    }
    \hfill
    \subfigure[Objective for Example 1]{
    \begin{minipage}[t]{0.4\linewidth}
        \centering
        \includegraphics[width=0.8\textwidth]{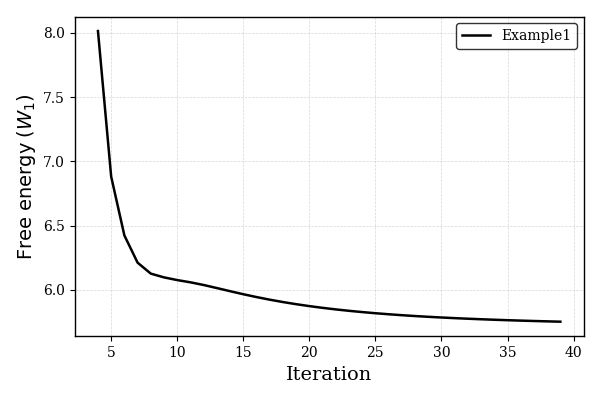}
        \label{DissipationEnergy1}
    \end{minipage}
    }
    \vspace{1em}
    \subfigure[Initial field]{
    \begin{minipage}[t]{0.31\linewidth}
        \centering
        \includegraphics[width=0.8\textwidth]{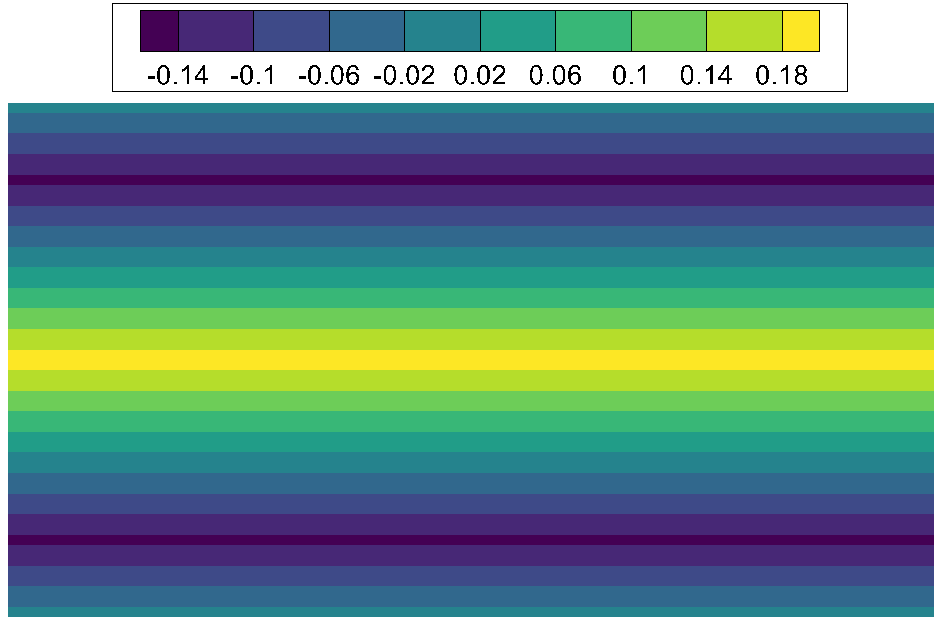}
    \end{minipage}
    }
    \hfill
    \subfigure[Pressure distribution]{
    \begin{minipage}[t]{0.31\linewidth}
        \centering
        \includegraphics[width=0.8\textwidth]{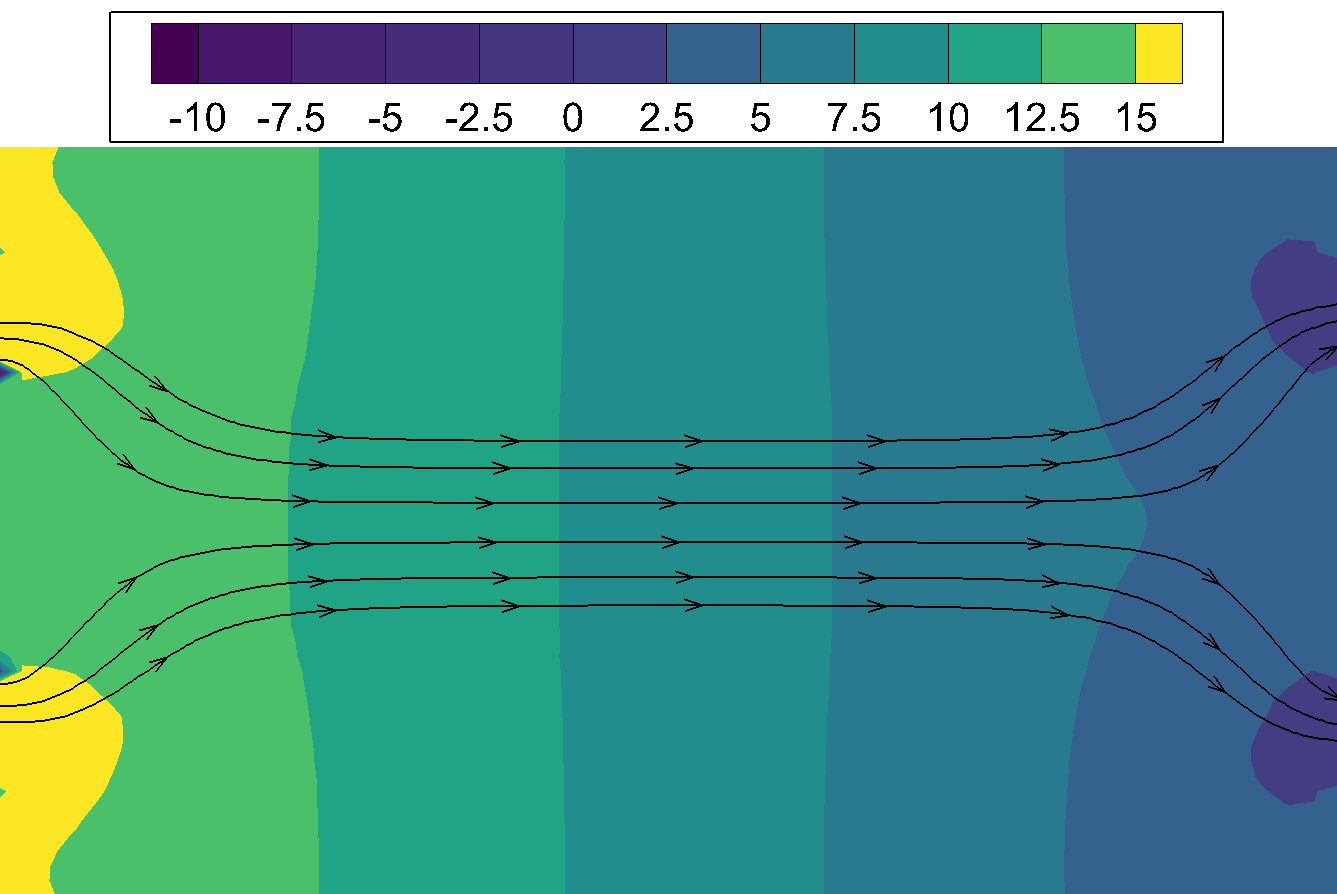}
        
        \label{Pressure distribution with the corresponding velocity profiles for example 1}
    \end{minipage}
}
    \hfill
    \subfigure[Optimized design with velocity]{
    \begin{minipage}[t]{0.31\linewidth}
        \centering
        \includegraphics[width=0.8\textwidth]{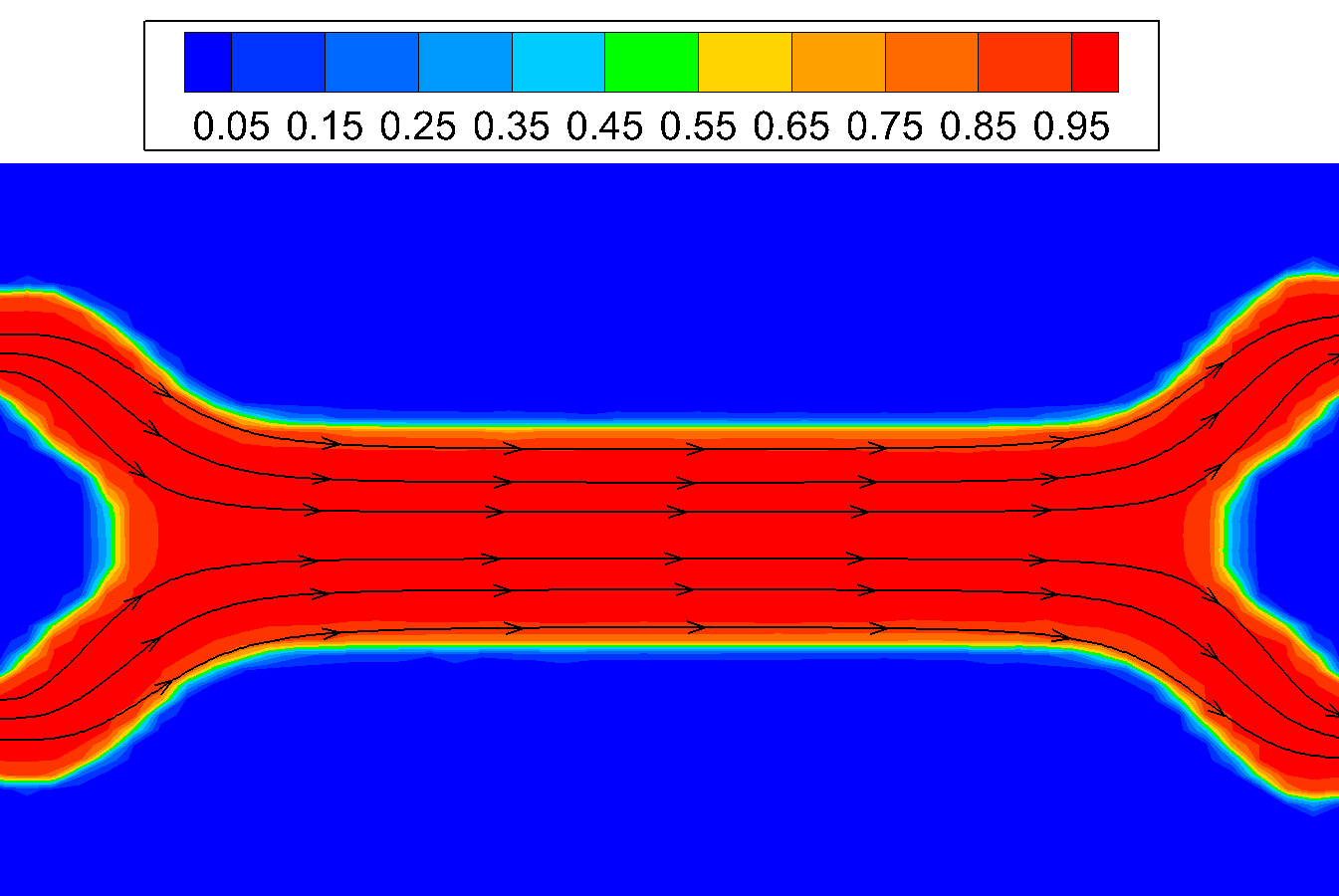}
        
        \label{Optimization distribution with the correspond-
ing velocity profiles for example 1}
    \end{minipage}%
    }
\label{EnergyDissiExample1}
\caption{Example 1 for the energy dissipation problem}
\end{figure}
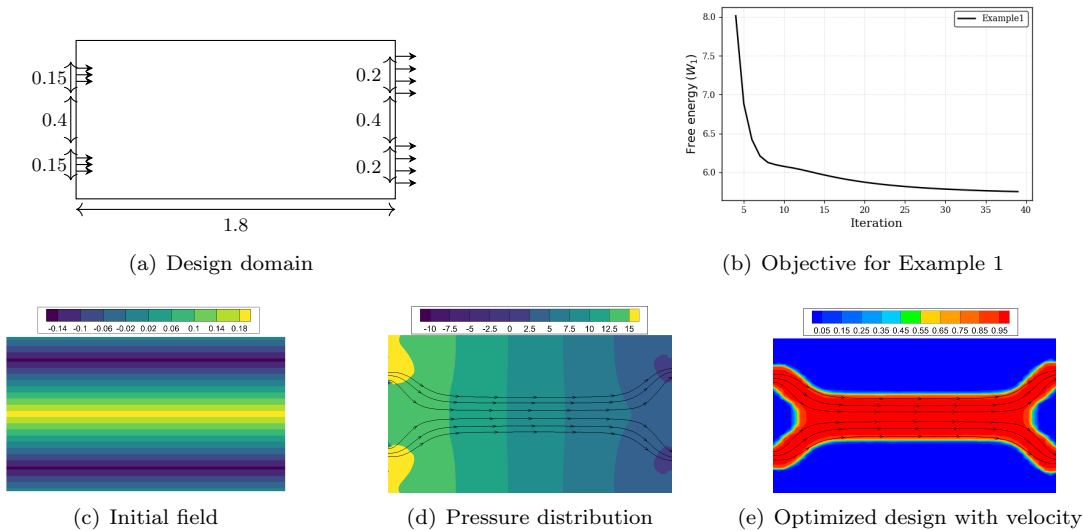

$\rm \bf Example\ 2:$ The second example is a shape design for the diffuser. The design domain is set to be a square $\Omega=[0,1]\times[-0.5,0.5]$ (see Fig. \ref{energy dissipation domain for example 2}). The flow on the inlet is imposed as ${\bm u_0}=[2(1-y)(1+y),0]^{\rm T}$, for which a parabolic inflow velocity profile is imposed on the left boundary. The homogeneous Neumann boundary condition of fluid is imposed on the outlet of the right boundary. We set the target volume of solid to be $V_0=0.35$ and the initial phase-field function to be $\phi_0(x,y)=\min\{|y-0.45|-0.225,|y+0.45|-0.225\}$. The parameters are set: ${\rm Re}=13.69, \alpha=10^2, \kappa=5\times 10^{-3}, {\tau}=5\times 10^{-3}, \eta=75, \beta_3=6\times 10^2, N=40$, the optimized distributions and the pressure contours are shown in Fig. \ref{Optimization distribution with the correspond-
ing velocity profiles for example 2} and Fig. \ref{Pressure distribution with the correspond-
ing velocity profiles for example 2}.

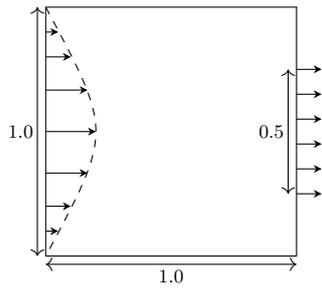
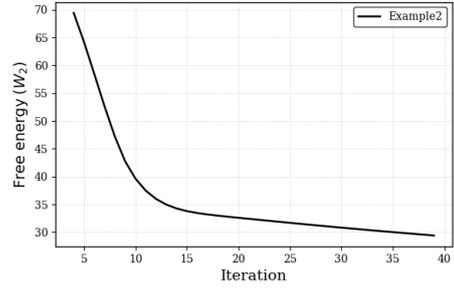
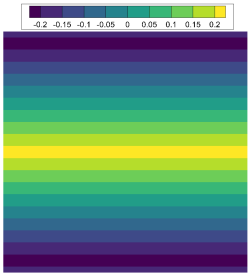
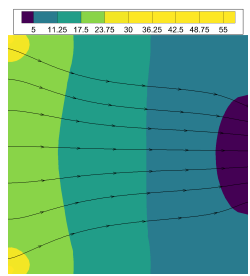
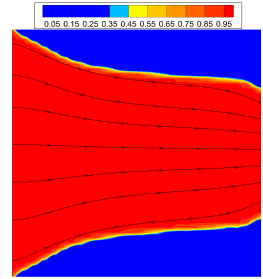
\begin{figure}[htbp]
\centering
\subfigure[Design domain]{
    \begin{minipage}[t]{0.25\linewidth}
    \centering
    
    \begin{tikzpicture}[yscale=1.1,xscale=1.1]
    \draw (0,0) rectangle (3,3);
    \draw[dashed] (0,3.0)..controls(0.8,1.5)..(0,0);
    \draw[arrows = {-Stealth[width=3pt, length=3pt]}](0,2.7)--(0.15,2.7);
    \draw[arrows = {-Stealth[width=3pt, length=3pt]}](0,2.4)--(0.3,2.4);
    \draw[arrows = {-Stealth[width=3pt, length=3pt]}](0,2.0)--(0.5,2.0);
    \draw[arrows = {-Stealth[width=3pt, length=3pt]}](0,1.5)--(0.6,1.5);
    \draw[arrows = {-Stealth[width=3pt, length=3pt]}](0,1.0)--(0.5,1.0);
    \draw[arrows = {-Stealth[width=3pt, length=3pt]}](0,0.6)--(0.3,0.6);
    \draw[arrows = {-Stealth[width=3pt, length=3pt]}](0,0.3)--(0.15,0.3);

    \draw[arrows = {-Stealth[width=3pt, length=3pt]}](3,2.25)--(3.3,2.25);
    \draw[arrows = {-Stealth[width=3pt, length=3pt]}](3,1.95)--(3.3,1.95);
    \draw[arrows = {-Stealth[width=3pt, length=3pt]}](3,1.65)--(3.3,1.65);
    \draw[arrows = {-Stealth[width=3pt, length=3pt]}](3,1.35)--(3.3,1.35);
    \draw[arrows = {-Stealth[width=3pt, length=3pt]}](3,1.05)--(3.3,1.05);
    \draw[arrows = {-Stealth[width=3pt, length=3pt]}](3,0.75)--(3.3,0.75);
    
    \draw[<->={-Stealth[width=3pt, length=3pt]}](0,-0.1)--(3,-0.1);
    \draw (1.5,-0.25) node[font=\small, scale=0.8]{$1.0$};

    \draw[<->={-Stealth[width=3pt, length=3pt]}](-0.1,0)--(-0.1,3.0);
    \draw (-0.3,1.5) node[font=\small, scale=0.8]{$1.0$};

    \draw[<->={-Stealth[width=3pt, length=3pt]}](2.9,0.75)--(2.9,2.25);
    \draw (2.7,1.5) node[font=\small, scale=0.8]{$0.5$};
    
    \end{tikzpicture}
     \label{energy dissipation domain for example 2}
 
    \end{minipage}%
}
    \hfill
    \subfigure[Objective for Example 2]{
    \begin{minipage}[t]{0.45\linewidth}
        \centering
        \includegraphics[width=0.9\textwidth]{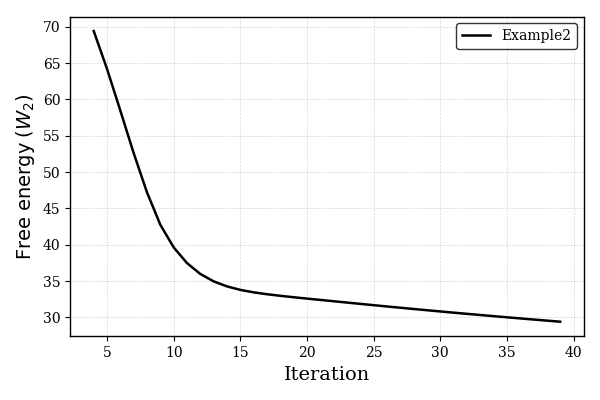}
        \label{DissipationEnergy2}
    \end{minipage}
    }
    \vspace{1em}
    \subfigure[Initial field]{
    \begin{minipage}[t]{0.31\linewidth}
        \centering
        \includegraphics[width=0.7\textwidth]{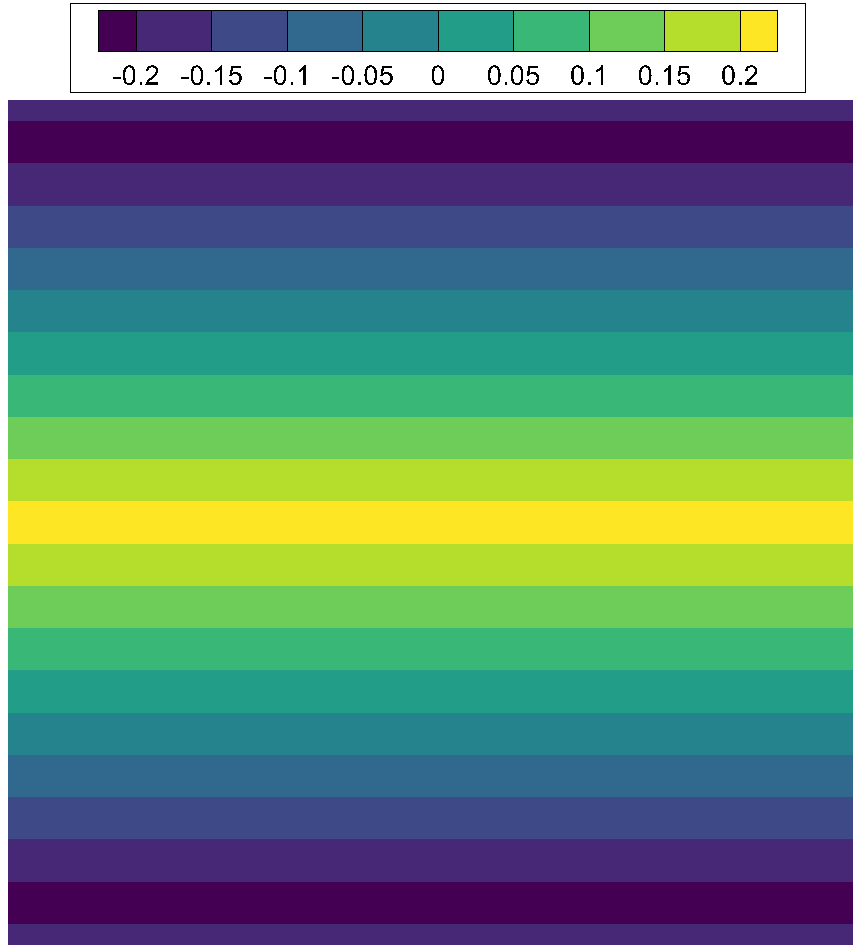}
    \end{minipage}
    }
    \hfill
    \subfigure[Pressure distribution]{
    \begin{minipage}[t]{0.3\linewidth}
        \centering
        \includegraphics[width=0.7\textwidth]{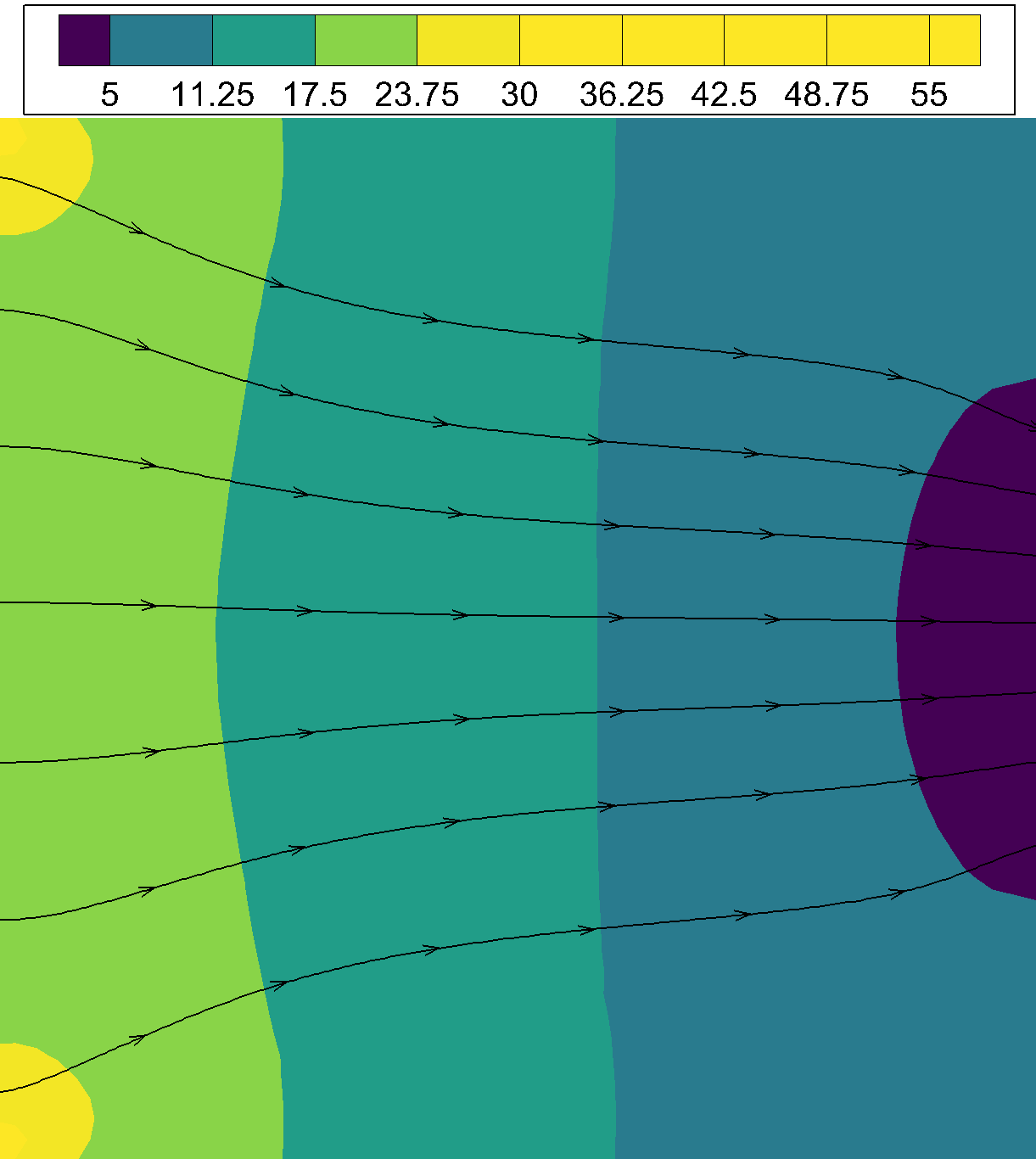}
        
        \label{Pressure distribution with the correspond-
ing velocity profiles for example 2}
    \end{minipage}
}
    \hfill
    \subfigure[Optimized design with velocity]{
    \begin{minipage}[t]{0.31\linewidth}
        \centering
        \includegraphics[width=0.7\textwidth]{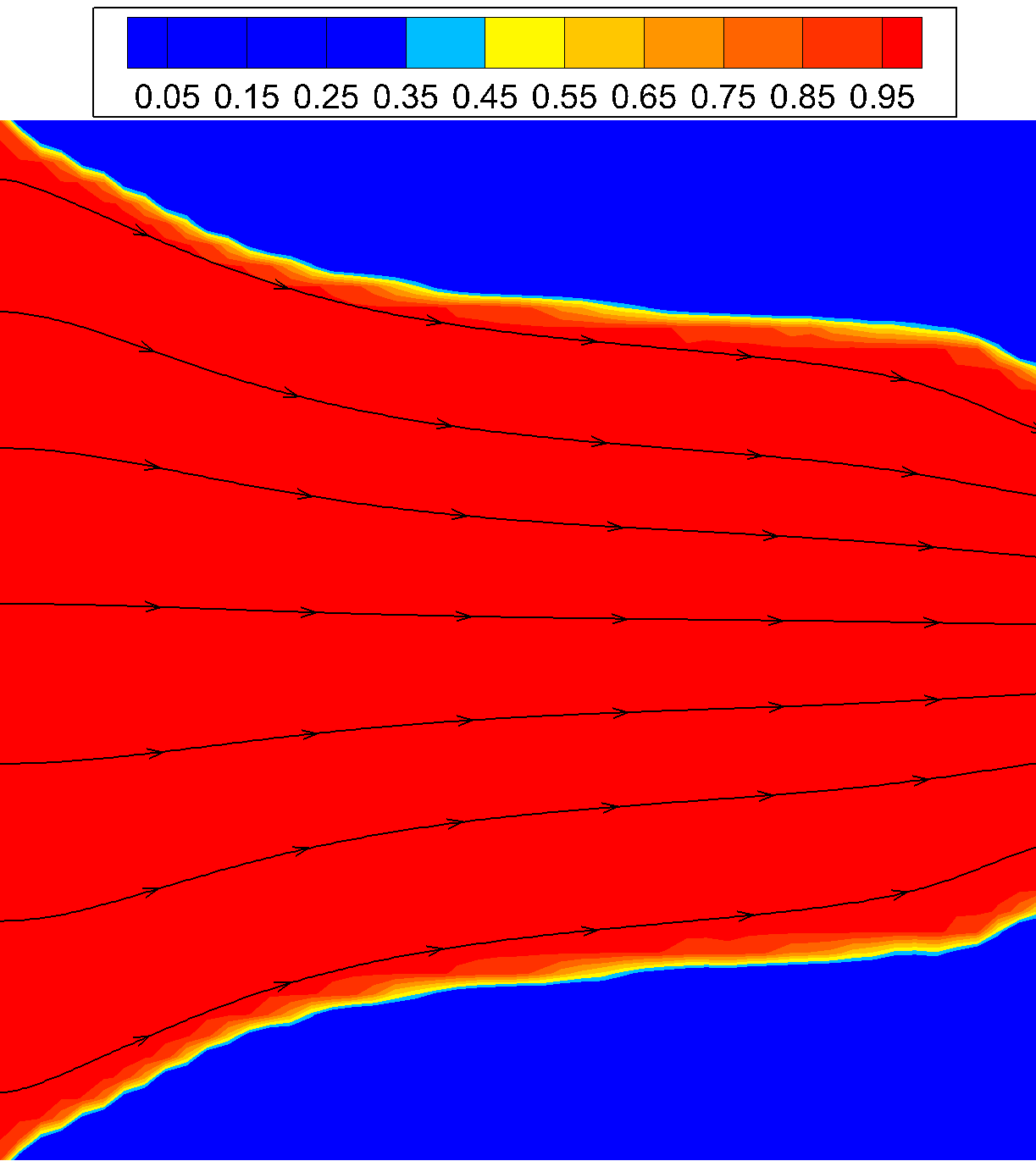}
        
        \label{Optimization distribution with the correspond-
ing velocity profiles for example 2}
    \end{minipage}%
    }
    \label{Example2}
\caption{Results for the energy dissipation problem: Example 2.}
\end{figure}

$\rm \bf Example\ 3:$ We perform the shape design for pipe bending in a square $\Omega=[0,1]\times[0,1]$ as the third example. The inflow fluid velocity is set to be ${\bm u_0}=[10(0.85-y)(y-0.5),0]^{\rm T}$ on the inlet and the target volume is $V_0=0.78$. We impose the parabolic inflow of fluid through an opening located at the upper-left side of the square domain, with an outflow occurring through another opening situated at the lower-right side (see Fig. \ref{EnergyDissipationDomainForExample3}). The initial phase-field function is set to be $\phi_0=1-|x+y-1|$ with the parameters: ${\rm Re=13.69}, \alpha=10^2, \kappa=5\times 10^{-3}, {\tau}=5\times 10^{-4}, \eta=75, \beta_3=6\times 10^2, N=40$. We present the optimized distributions with the corresponding velocity profile and pressure contour shown in Fig. \ref{Optimization distribution with the correspond-
ing velocity profiles for example 3} and Fig. \ref{Pressure distribution with the correspond-
ing velocity profiles for example 3}.

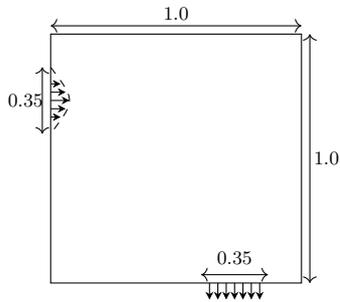
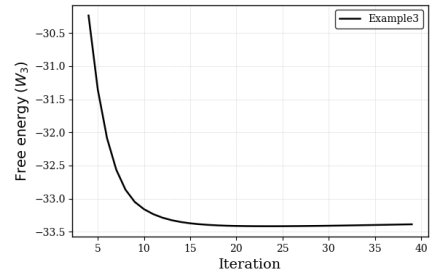
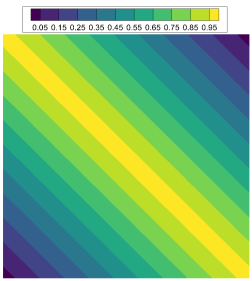
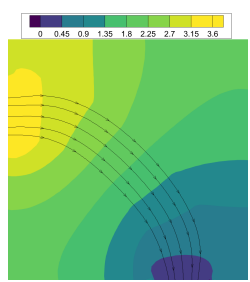
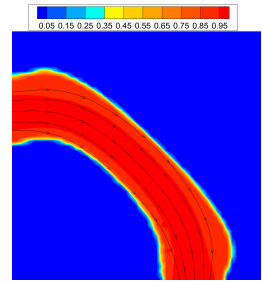
\begin{figure}[htbp]
\centering
\subfigure[Design domain]{
    \begin{minipage}[t]{0.25\linewidth}
    \centering
    \begin{tikzpicture}[yscale=1.1,xscale=1.1]
    \draw (0,0) rectangle (3,3);
    \draw[dashed] (0,2.6)..controls(0.3,2.2)..(0,1.8);
    \draw[arrows = {-Stealth[width=3pt, length=3pt]}](0,2.4)--(0.12,2.4);
    \draw[arrows = {-Stealth[width=3pt, length=3pt]}](0,2.3)--(0.18,2.3);
    \draw[arrows = {-Stealth[width=3pt, length=3pt]}](0,2.2)--(0.225,2.2);
    \draw[arrows = {-Stealth[width=3pt, length=3pt]}](0,2.1)--(0.18,2.1);
    \draw[arrows = {-Stealth[width=3pt, length=3pt]}](0,2.0)--(0.12,2.0);

    \draw[arrows = {-Stealth[width=3pt, length=3pt]}](2.5,0)--(2.5,-0.2);
    \draw[arrows = {-Stealth[width=3pt, length=3pt]}](2.4,0)--(2.4,-0.2);
    \draw[arrows = {-Stealth[width=3pt, length=3pt]}](2.3,0)--(2.3,-0.2);
    \draw[arrows = {-Stealth[width=3pt, length=3pt]}](2.2,0)--(2.2,-0.2);
    \draw[arrows = {-Stealth[width=3pt, length=3pt]}](2.1,0)--(2.1,-0.2);
    \draw[arrows = {-Stealth[width=3pt, length=3pt]}](2.0,0)--(2.0,-0.2);
    \draw[arrows = {-Stealth[width=3pt, length=3pt]}](1.9,0)--(1.9,-0.2);

    \draw[<->={-Stealth[width=3pt, length=3pt]}](0,3.1)--(3,3.1);
    \draw (1.5,3.25) node[font=\small, scale=0.8]{$1.0$};

    \draw[<->={-Stealth[width=3pt, length=3pt]}](3.1,0)--(3.1,3.0);
    \draw (3.3,1.5) node[font=\small, scale=0.8]{$1.0$};

    \draw[<->={-Stealth[width=3pt, length=3pt]}](-0.1,1.8)--(-0.1,2.6);
    \draw (-0.3,2.2) node[font=\small, scale=0.8]{$0.35$};

    \draw[<->={-Stealth[width=3pt, length=3pt]}](1.8,0.1)--(2.6,0.1);
    \draw (2.2,0.3) node[font=\small, scale=0.8]{$0.35$};
    
    
    \end{tikzpicture}
     \label{EnergyDissipationDomainForExample3}

    \end{minipage}%
}
    \hfill
    \subfigure[Objective for Example 3]{
    \begin{minipage}[t]{0.45\linewidth}
        \centering
        \includegraphics[width=0.85\textwidth]{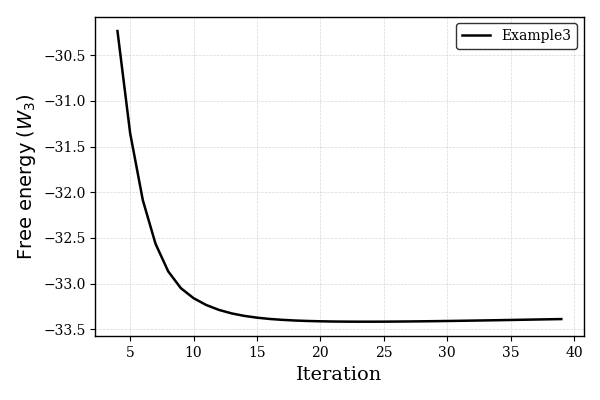}
        \label{DissipationEnergy3}
    \end{minipage}
    }
    \subfigure[Initial field]{
    \begin{minipage}[t]{0.31\linewidth}
        \centering
        \includegraphics[width=0.7\textwidth]{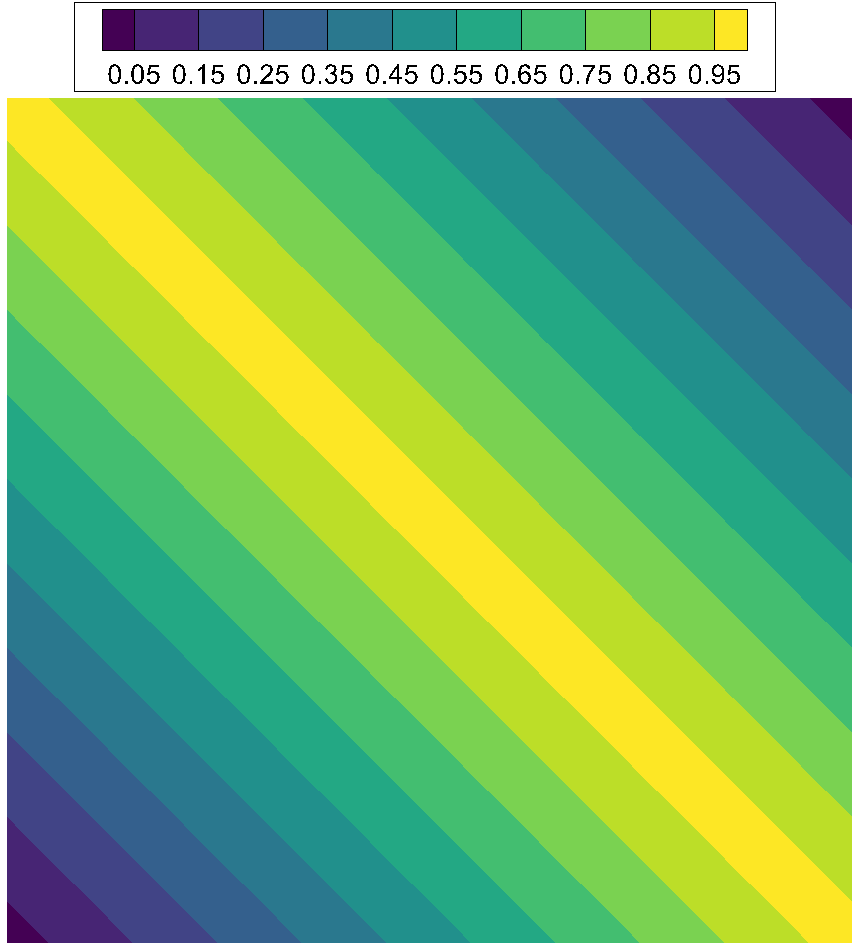}
    \end{minipage}
    }
    \hfill
    \subfigure[Pressure distribution]{
    \begin{minipage}[t]{0.3\linewidth}
        \centering
        \includegraphics[width=0.7\textwidth]{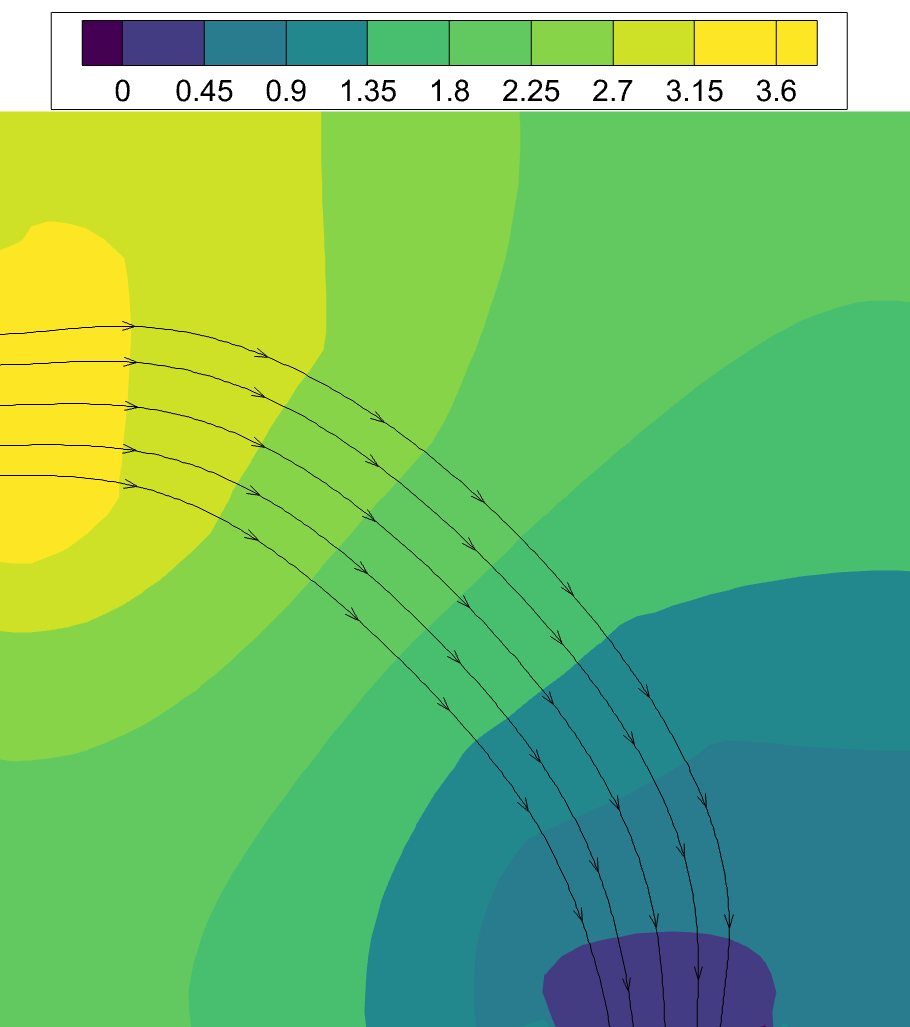}
        
        \label{Pressure distribution with the correspond-
ing velocity profiles for example 3}
    \end{minipage}
}
    \hfill
    \subfigure[Optimized design with velocity]{
    \begin{minipage}[t]{0.31\linewidth}
        \centering
        \includegraphics[width=0.7\textwidth]{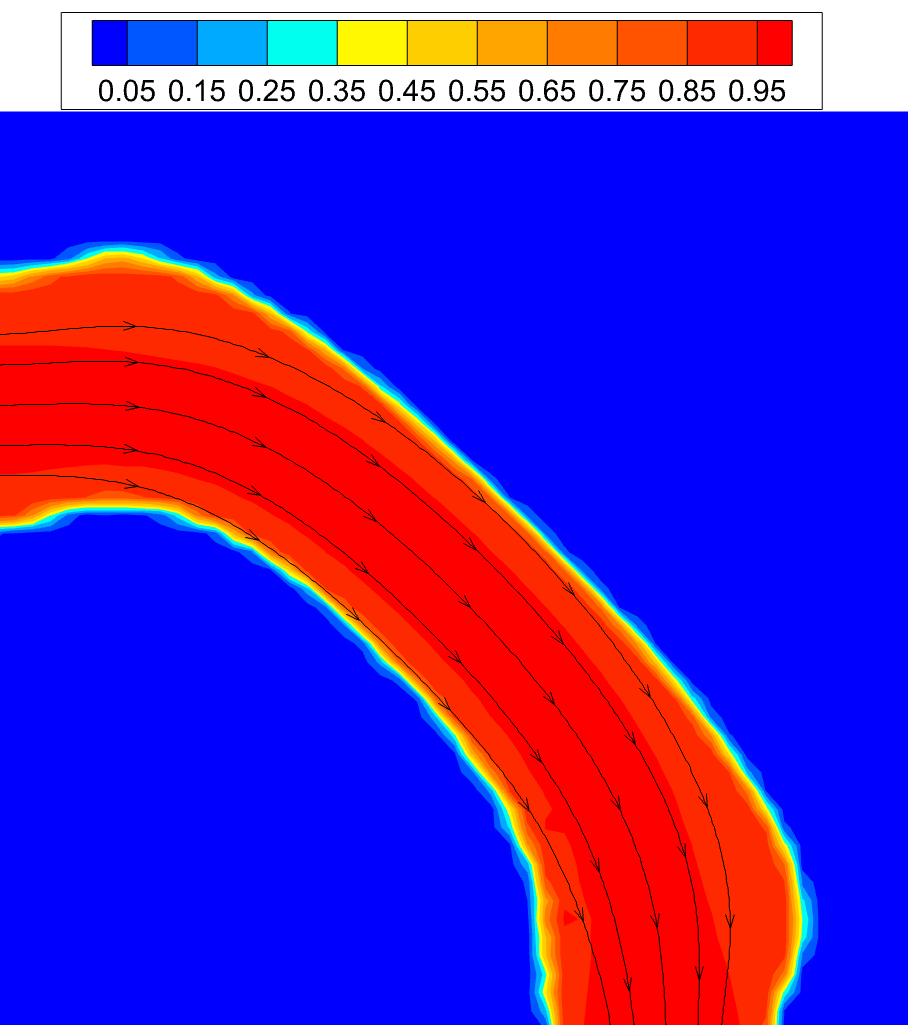}
        
        \label{Optimization distribution with the correspond-
ing velocity profiles for example 3}
    \end{minipage}%
    }
    \label{EnergyDissiExample3}
\caption{Results for the energy dissipation problem: Example 3.}
\end{figure}

In the above three energy dissipation problems, all computations are performed using a discretization of finite element mesh with $125 \times 150$ triangular elements, and the evolution of the phase-field is governed by a modified Ginzburg–Landau free energy functional, which is consistently expressed as \eqref{W1}. In addition, we present the convergence histories of free energy in Fig. \ref{DissipationEnergy1}, \ref{DissipationEnergy2} and \ref{DissipationEnergy3}, respectively. Both the free energy of the phase-field function and the objective function tend to be flat at the end, and the volume errors for the three examples are limited to $3\%$ of the total volume during the whole iteration.

The results show that the objective function decreases smoothly to a minimum, demonstrating the effectiveness of the proposed modified formulation across different optimization scenarios. To further validate the stability of the algorithm, we next apply it to a more complex three-dimensional problem.


Next we test the topology optimization problem of 3D bypass. The example is solved by Algorithm \ref{algPB} and also omits the Poisson-Boltzmann and the convection–diffusion equations. The design domains are presented in Fig. \ref{DesignDomainFor3dEnergyDissi}.


\textbf{Example 4}: This is a three-dimensional version of Example 1. Fluid enters the domain through two symmetric inlets located on the left boundary and exits through two symmetric outlets on the right boundary. The design domain is a rectangular box of size $1.8 \times 1 \times 0.3$ (see Fig. \ref{EnergyDissi3dS1}).

The inlet velocity is prescribed as ${\bm u}_0 = [50(0.35 - y)(0.35 + y),\ 0,\ 0]^{\rm T}$, representing a symmetric parabolic inflow profile. No-slip boundary conditions are imposed on all other boundaries. Set $V_0 = 0.45$. The initial phase-field function is set as
$\phi_0 = \min\left( \left| y - 0.35 \right| - 0.15,\ \left| y + 0.35 \right| - 0.15 \right),$ which initializes two symmetric solid bars in the design domain.
Other relevant parameters are listed: ${\rm Re=13.69}, \alpha=2\times 10^2, \kappa= 10^{-4}, {\tau}=2\times 10^{-5}, \eta=10^3, \beta_3=8\times 10^2, N=60$. The optimized design and its velocity field are shown in Fig. \ref{EnergyDissi3dS2} and Fig. \ref{EnergyDissi3dV2}, respectively. We present the convergence history of free energy in Fig. \ref{ConvergenceHistoriesOfEnergyDissipation3d}. 

\begin{figure}[htbp]
    \centering
    
    \subfigure[Design domain for 3D energy dissipation problem]{
    \begin{minipage}[t]{0.5\linewidth} 
        \centering
    \includegraphics[width=0.8\textwidth]{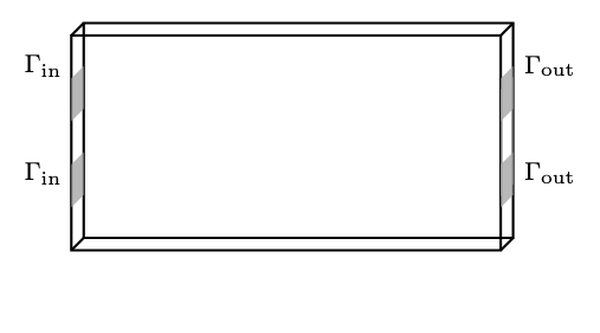}
            \label{EnergyDissi3dS1}
    \end{minipage}
}
    \hfill
    \subfigure[Convergence histories of objective]{
    \begin{minipage}[t]{0.45\linewidth}
        \centering
        \includegraphics[width=0.8\textwidth]{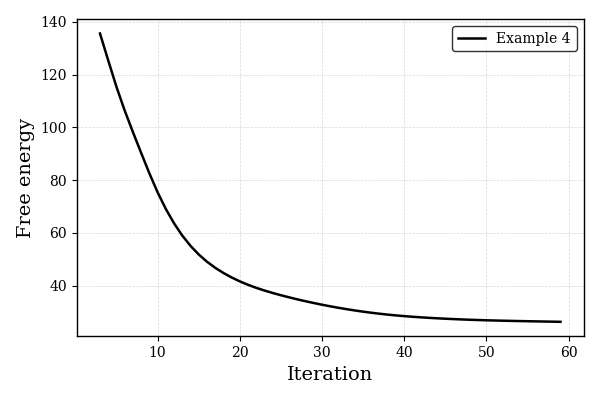}
        \label{ConvergenceHistoriesOfEnergyDissipation3d}
    \end{minipage}%
    }
    \vspace{1em}
    \subfigure[Optimized phase-field function]{
    \begin{minipage}[t]{0.45\linewidth}
        \centering
        \includegraphics[width=0.8\textwidth]{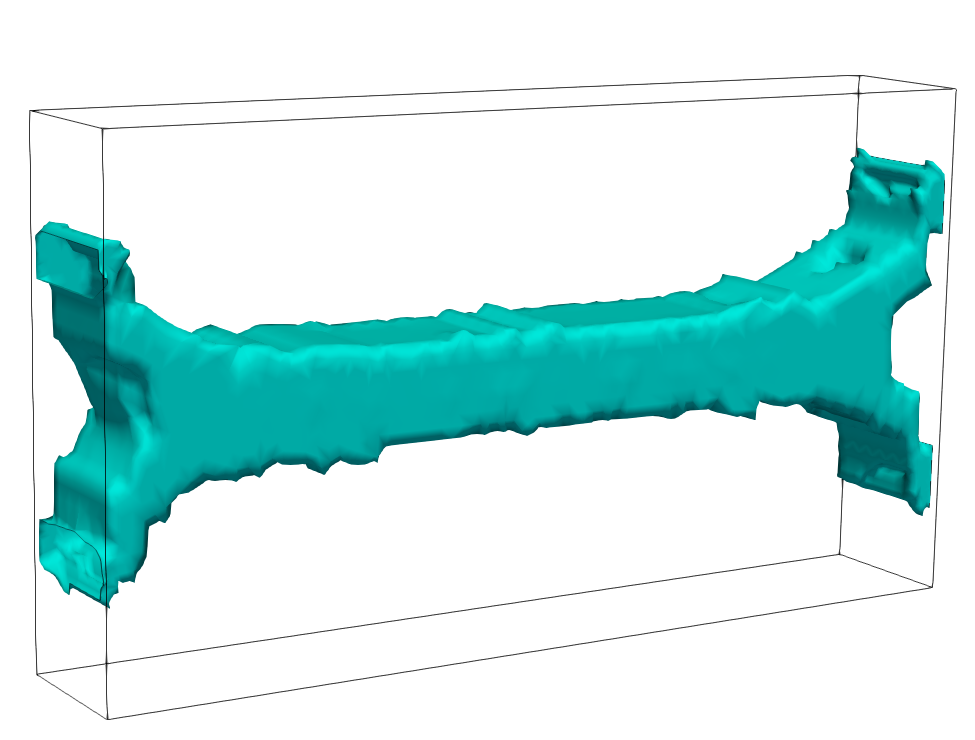}
        \label{EnergyDissi3dS2}
    \end{minipage}
    }
    \hfill
    \subfigure[Velocity field on optimized shape]{
    \begin{minipage}[t]{0.45\linewidth}
        \centering
        \includegraphics[width=0.8\textwidth]{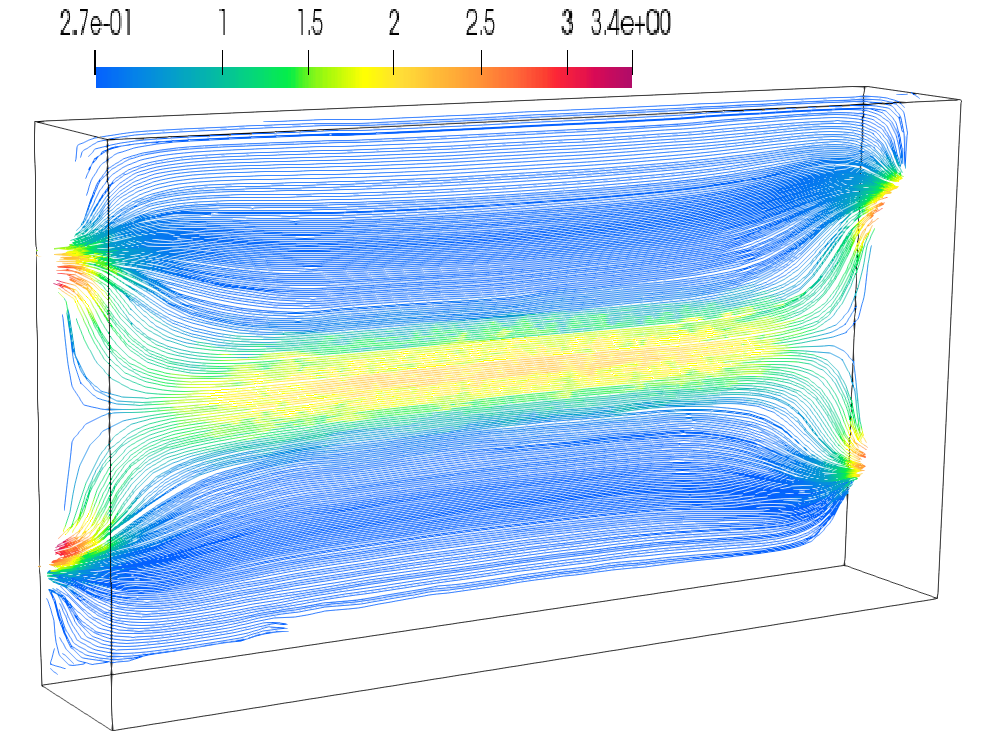}
        \label{EnergyDissi3dV2}
    \end{minipage}
    }
    \caption{Design domain, convergence history of energy, optimized phase-field, and velocity distribution for Example 4.}
    \label{DesignDomainFor3dEnergyDissi}
\end{figure}

\subsection{Microfluidic mixer problem}

We now apply the proposed scheme to the micromixer problem with multi-physics to evaluate its performance under multiphysics coupling. 
\subsubsection{2D mixer problem}

The micromixing mixing problem was first solved by with topology optimization \cite{Andreasen2009}. Two incompressible fluids labeled $c = 0$ and $c = 1$ are introduced from the upper and lower parts of the inlet boundary, respectively. By optimizing the phase-field distribution within the rectangular domain, we aim to enhance the mixing performance and achieve a uniform concentration distribution at the outlet.

To prevent inlet blockage, which is commonly encountered in micromixing problems, placing the design domain in the central part of the channel, with buffer regions fixed as fluid near both the inlet and the outlet to ensure unobstructed flow was proposed \cite{Andreasen2009}. We follow this configuration in our setup as illustrated in Fig. \ref{AndreasenDesignDomain}, and in the following three examples, a finite element mesh consisting of $200 \times 100$ triangular elements is adopted.
\begin{figure}[htbp]
    \centering
    \begin{tikzpicture}[yscale=0.6,xscale=2.2]
\draw (0,0) rectangle (3,3);
\draw[dashed] (0,3)..controls(0.8,1.5)..(0,0);
\draw[<->={-Stealth[width=3pt, length=3pt]}] (-0.05,0)--(-0.05,1.5);
\draw(-0.2,0.75) node[font=\small, scale=0.7]{$c=0$};
\draw[<->={-Stealth[width=3pt, length=3pt]}] (-0.05,3)--(-0.05,1.5);
\draw(-0.2,2.25) node[font=\small, scale=0.7]{$c=1$};
\draw[arrows = {-Stealth[width=3pt, length=3pt]}](0,0.5)--(0.25,0.5);
\draw[arrows = {-Stealth[width=3pt, length=3pt]}](0,2.5)--(0.25,2.5);
\draw[arrows = {-Stealth[width=3pt, length=3pt]}](0,1)--(0.5,1);
\draw[arrows = {-Stealth[width=3pt, length=3pt]}](0,2)--(0.5,2);
\draw[arrows = {-Stealth[width=3pt, length=3pt]}](0,1.5)--(0.6,1.5);
\draw[arrows = {-Stealth[width=3pt, length=3pt]}](3,0.5)--(3.25,0.5);
\draw[arrows = {-Stealth[width=3pt, length=3pt]}](3,1)--(3.25,1);
\draw[arrows = {-Stealth[width=3pt, length=3pt]}](3,1.5)--(3.25,1.5);
\draw[arrows = {-Stealth[width=3pt, length=3pt]}](3,2)--(3.25,2);
\draw[arrows = {-Stealth[width=3pt, length=3pt]}](3,2.5)--(3.25,2.5);
\draw[<->={-Stealth[width=3pt, length=3pt]}] (2.95,0)--(2.95,3);

\fill[gray!70] (0.7,0) rectangle (2.3,3);
\node at (1.5,1.5) {\textbf{\textcolor{white}{Design domain}}};

\draw(2.7,1.5) node[font=\small, scale=0.7]{$c_d=0.5$};

\draw(1.5,0.25) node[font=\small, scale=0.7]{$\Gamma_d$};
\draw(1.5,2.75) node[font=\small, scale=0.7]{$\Gamma_u$};
	\end{tikzpicture}

\caption{Design domain for mixer problem}
\label{AndreasenDesignDomain}
\end{figure}
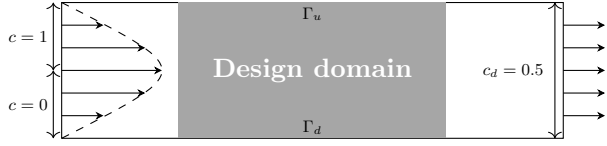
Here, we set the design domain to be about $80\%$ of the whole domain and the prescribed concentration function is
\begin{equation}\label{InitialConcentration}
\begin{aligned}
c_0(x,y) = 
\begin{cases}
    0, & \text{if } y < 0.5, \\
    1, & \text{if } y \ge 0.5.\\
\end{cases}
\end{aligned}
\end{equation} 
Set $\beta_1 = 0$, $\beta_2 = 1$, $\beta_3 = 20$, $N_p=10$, and $N=100$.

\textbf{Example 5:}
Consider a domain of size $4 \times 1$. Set the initial
\begin{equation}\label{InitialShapeFor4times1}
\begin{aligned}
\phi_0(x, y) =
\begin{cases}
    1, & \text{if } x < 0.5 \text{ or } x > 3.625, \\
    1 - \left( 0.5 \cos(4\pi x) \cos\left(2\pi y + \frac{\pi}{2} \right) + 0.15 \right), & \text{if } 0.5 \le x \le 3.625,
\end{cases}
\end{aligned}
\end{equation}
which is shown in Fig. \ref{InitialShapeMixer2d}.
Set the inflow velocity ${\bm u_0} = [0.2y(1 - y),0]^{\rm T}$, and the target volume $V_0 = \frac{15}{32}$. We apply Algorithm \ref{algPB}, in which the Newton iteration for the Poisson-Boltzmann equations is skipped, using the following parameters: ${\rm Re} = 1.0$, ${\rm Pe} = 3\times 10^2$, $\alpha_0 = 8\times 10^2$, ${\tau} = 8\times10^{-4}$, $\kappa = 10^{-3}$, and $\eta = 2$. The optimized results are shown in Fig. \ref{NSCD2dC}, Fig. \ref{NSCD2dS} and Fig. \ref{NSCD2dV}. The iteration process of the free energy is shown in Fig. \ref{NSCDfreeenergy}.

\textbf{Example 6:} 
For comparison and broader applicability, a $2.5 \times 1$ domain is adopted. 
Using the same parameters as the former example and setting $V_0=\frac{9}{32}$, the optimized results are presented in Fig. \ref{NSCD2d2}. 
The modified free energy is formulated as \eqref{W2}.

The iteration process of the free energy is shown in Fig. \ref{NSCDFreeEnergy2}.
\begin{figure}[htbp]
\centering
\subfigure[Initial phase-field distribution for Example 5]{
    \begin{minipage}[t]{0.45\linewidth}
        \centering
        \includegraphics[width=0.93\textwidth]{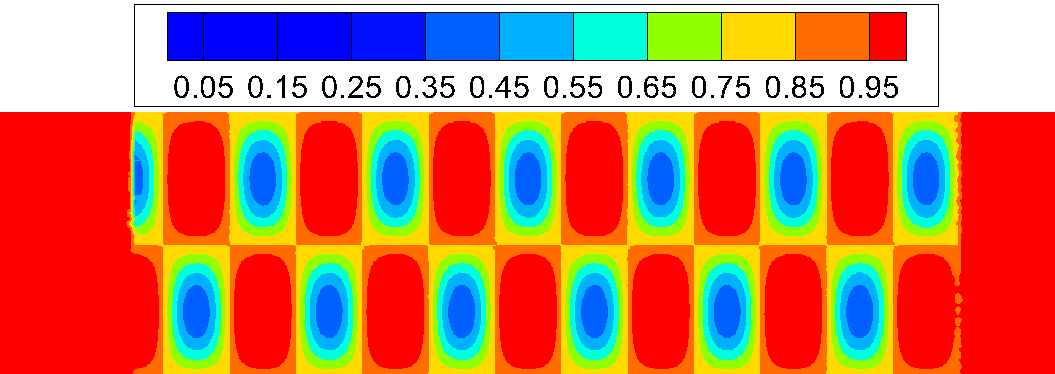}
        \label{InitialShapeMixer2d}
    \end{minipage}%
    }
    \hfill
    \subfigure[Optimized concentration distribution]{
    \begin{minipage}[t]{0.45\linewidth}
        \centering
        \includegraphics[width=0.93\textwidth]{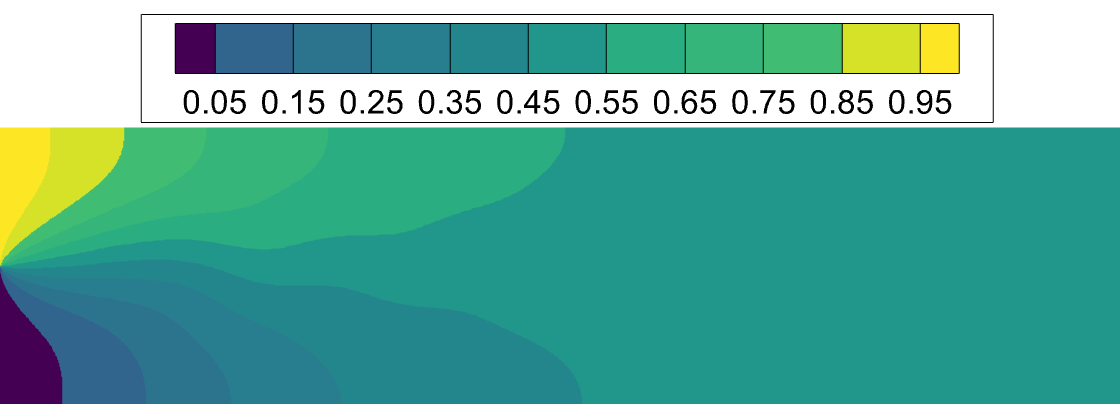}
        \label{NSCD2dC}
    \end{minipage}
    }
    \vspace{1em}
    \subfigure[Optimized phase-field distribution]{
    \begin{minipage}[t]{0.45\linewidth}
        \centering
        \includegraphics[width=0.93\textwidth]{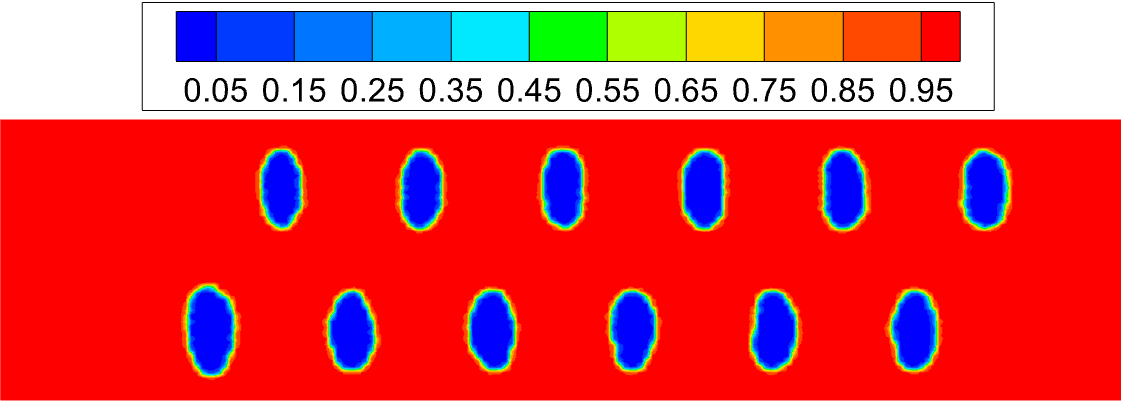}
        \label{NSCD2dS}
    \end{minipage}
    }
    \hfill
    \subfigure[Optimized velocity field]{
    \begin{minipage}[t]{0.45\linewidth}
        \centering
        \includegraphics[width=0.93\textwidth]{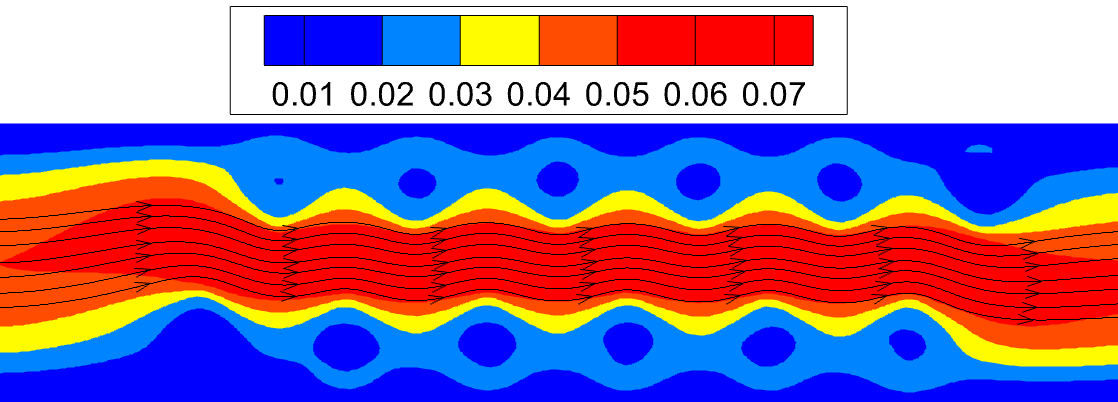}
        \label{NSCD2dV}
    \end{minipage}
    }
    \caption{Initial phase-field distribution and optimized results for Example 5.}
    \label{NSCD2d}
\end{figure}

\begin{figure}[htbp]
\centering
\subfigure[Initial phase-field distribution]{
    \begin{minipage}[t]{0.45\linewidth}
        \centering
        \includegraphics[width=0.8\textwidth]{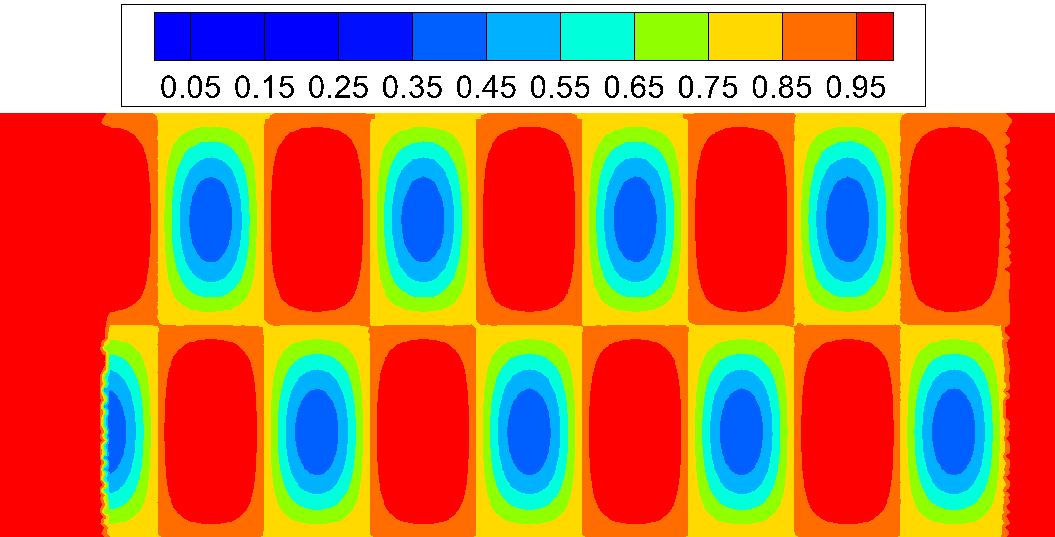}
        \label{InitialShapeMixer2d2}
    \end{minipage}%
    }
    \hfill
    \subfigure[Optimized concentration distribution]{
    \begin{minipage}[t]{0.45\linewidth}
        \centering
        \includegraphics[width=0.8\textwidth]{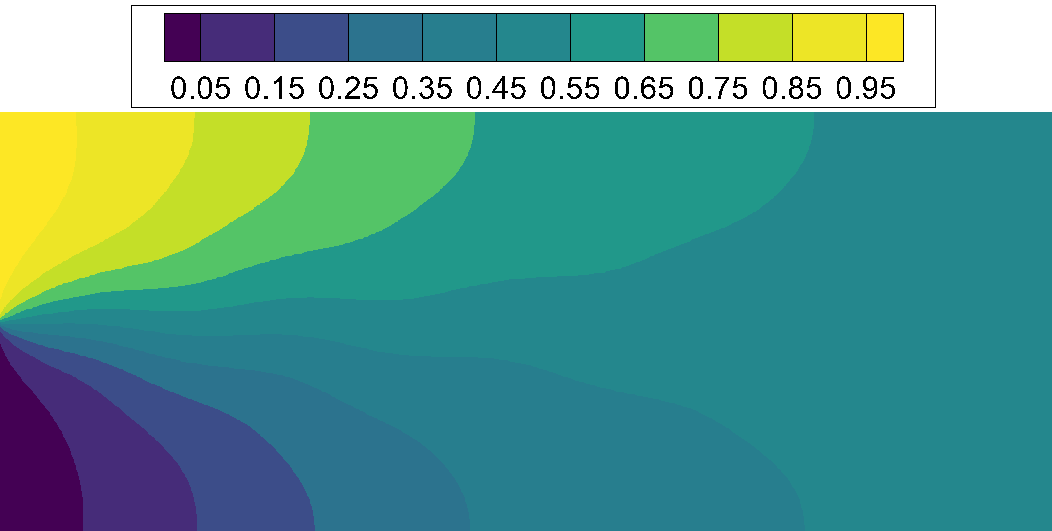}
        \label{NSCD2dC2}
    \end{minipage}
    }
    \vspace{1em}
    \subfigure[Optimized phase-field distribution]{
    \begin{minipage}[t]{0.45\linewidth}
        \centering
        \includegraphics[width=0.8\textwidth]{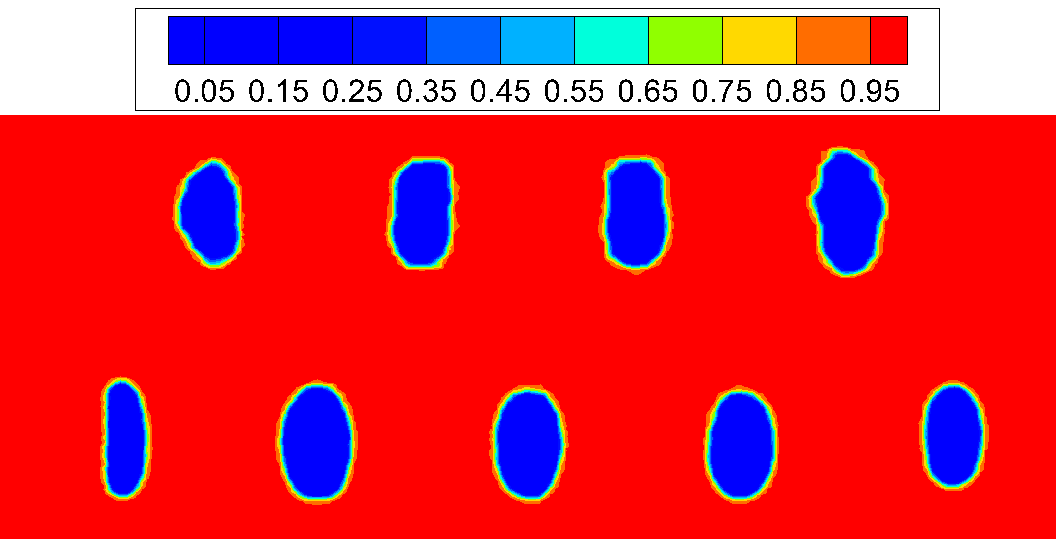}
        \label{NSCD2dS2}
    \end{minipage}
    }
    \hfill
    \subfigure[Optimized velocity field]{
    \begin{minipage}[t]{0.45\linewidth}
        \centering
        \includegraphics[width=0.8\textwidth]{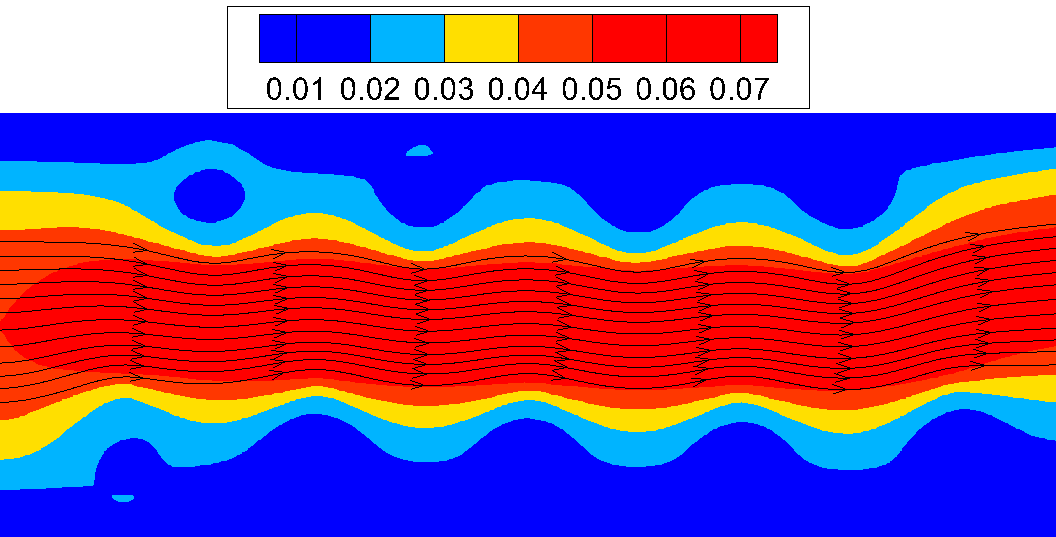}
        \label{NSCD2dV2}
    \end{minipage}
    }
    \caption{Initial phase-field distribution and optimized results for Example 6.}
    \label{NSCD2d2}
\end{figure}

As shown in Fig. \ref{NSCD2d} and Fig. \ref{NSCD2d2}, both configurations lead to the conclusion that efficient mixing requires internal solid structures to continuously redirect the flow, effectively confining it to the central region of the channel. Moreover, the optimized result for the longer channel appears to be an extension of the shorter one, further demonstrating the robustness of the algorithm and its potential scalability to larger domains. The decay for the mixing objective $J_2$ is shown in Fig. \ref{NSCDFreeEnergyIteration} (right). 

\begin{figure}[htbp]
\centering
\subfigure{
    \begin{minipage}[htbp]{0.45\linewidth}
        \centering
        \includegraphics[width=0.8\textwidth]{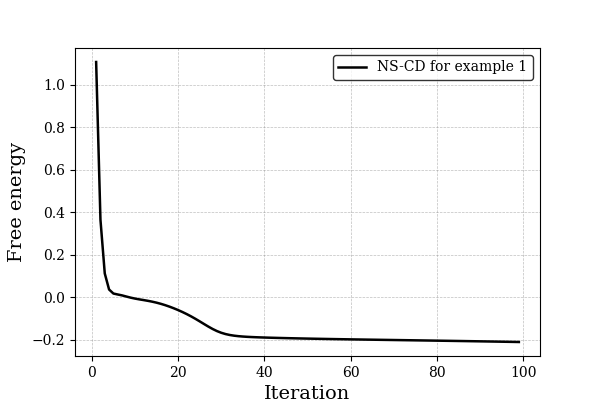}
        \label{NSCDfreeenergy}
    \end{minipage}%
    }
    \hfill
\subfigure{
    \begin{minipage}[htbp]{0.45\linewidth}
        \centering
        \includegraphics[width=0.8\textwidth]{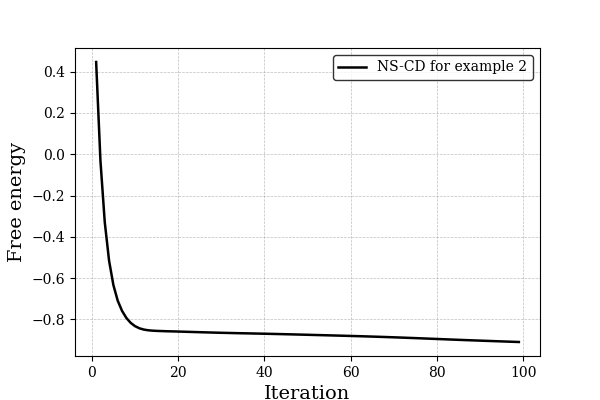}
        \label{NSCDFreeEnergy2}
    \end{minipage}
    }
    \caption{Convergence histories of objective for mixer problem: Example 5 (left) and Example 6 (right).}
    \label{NSCDFreeEnergyIteration}
\end{figure}

\textbf{Example 7:} To further evaluate the performance of our method in a complex multiphysics-coupled system, we consider an electrokinetically enhanced micromixing problem. The design domain is adopted from \cite{Ji2017}, as illustrated in Fig. \ref{DesignDomainForPBMixerProblem}. In this case, a periodic electric field is applied on both sides of the design domain, acting as an external force to enhance internal mixing and energy focusing. The solid region is treated as an insulating material. To model the electrostatic interaction, we couple the system with the Poisson–Boltzmann equation, allowing the electric field to guide the flow distribution while maintaining smooth inflow conditions. Here, a new boundary setting is adopted for the mixer problem in rectangle $\Omega=[0,4]\times[0,1]$. In this case, we set the electrical potential on the boundary to be
\begin{equation*}
\begin{cases}
\psi(x, y)=0.1y(1-y)^2\sin(\pi y), 
  & \text{on }\Gamma_i,\\
\psi(x, y) =0, 
  & \text{on } \Gamma_u\cup\Gamma_d,\\
\psi(x, y) =0.7(x-0.5)(x-1)(x-2)(x-3)\Bigl(x-\tfrac{29}{8}\Bigr)\sin{\pi x}, 
  & \text{on }\Gamma_u,\\
\psi(x, y) =0, 
  & \text{on }\Gamma_d,\\
  \nabla\psi\cdot\bm n = 0, 
  & \text{on }\Gamma_o.
\end{cases}
\end{equation*}
We use the same concentration function and the initial phase-field function as presented in \eqref{InitialConcentration} and \eqref{InitialShapeFor4times1}. The inflow velocity is set to be ${\bm u_0} = [0.2y(1 - y),0]^{\rm T}$, and the target volume is $V_0 = \frac{15}{32}$. We apply the Algorithm \ref{algPB}, in which the Newton iteration for the Poisson-Boltzmann equation is adopted, using the following parameters: $\varepsilon_0=3.6$, $\rho_0=1.8$, ${\rm Re} = 1.0$, ${\rm Pe} = 3\times 10^2$, $\alpha_0 = 8\times 10^2$, ${\tau} = 8\times 10^{-4}$, $\kappa = 10^{-3}$, and $\eta = 2$. The optimized results are shown in Fig. \ref{ResultsOfTheMixerProblemUnderElectricity}. The modified free energy applied is
$$\mathcal{W}_2 =\int_{\Omega} \frac{\kappa}{2} |\nabla \phi|^2 + \frac{1}{4} \phi^2 (1 - \phi)^2 + \eta \, \phi^3 (6\phi^2 - 15\phi + 10) \frac{J_2'}{\|J_2'\|}\dx+\frac{\beta_3}{2}(V(\phi)-V_0)^2.$$
And the iteration process of the free energy is shown in Fig. \ref{ConvergenceHistoriesOfMixer3d} (left).

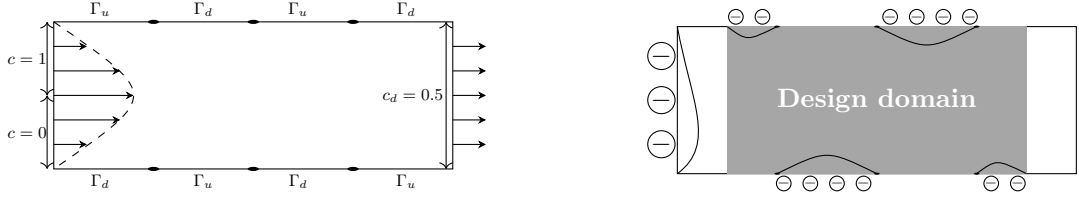
\begin{figure}[htbp]
\centering
\subfigure{
    \begin{minipage}[t]{0.45\linewidth}
    \centering
    \begin{tikzpicture}[yscale=0.65,xscale=1.75]
\draw (0,0) rectangle (3,3);
\draw[dashed] (0,3)..controls(0.8,1.5)..(0,0);
\draw[<->={-Stealth[width=3pt, length=3pt]}] (-0.05,0)--(-0.05,1.5);
\draw(-0.2,0.75) node[font=\small, scale=0.7]{$c=0$};
\draw[<->={-Stealth[width=3pt, length=3pt]}] (-0.05,3)--(-0.05,1.5);
\draw(-0.2,2.25) node[font=\small, scale=0.7]{$c=1$};
\draw[arrows = {-Stealth[width=3pt, length=3pt]}](0,0.5)--(0.25,0.5);
\draw[arrows = {-Stealth[width=3pt, length=3pt]}](0,2.5)--(0.25,2.5);
\draw[arrows = {-Stealth[width=3pt, length=3pt]}](0,1)--(0.5,1);
\draw[arrows = {-Stealth[width=3pt, length=3pt]}](0,2)--(0.5,2);
\draw[arrows = {-Stealth[width=3pt, length=3pt]}](0,1.5)--(0.6,1.5);
\draw[arrows = {-Stealth[width=3pt, length=3pt]}](3,0.5)--(3.25,0.5);
\draw[arrows = {-Stealth[width=3pt, length=3pt]}](3,1)--(3.25,1);
\draw[arrows = {-Stealth[width=3pt, length=3pt]}](3,1.5)--(3.25,1.5);
\draw[arrows = {-Stealth[width=3pt, length=3pt]}](3,2)--(3.25,2);
\draw[arrows = {-Stealth[width=3pt, length=3pt]}](3,2.5)--(3.25,2.5);
\draw[<->={-Stealth[width=3pt, length=3pt]}] (2.95,0)--(2.95,3);
\draw(2.7,1.5) node[font=\small, scale=0.7]{$c_d=0.5$};
\draw(0.35,3.25) node[font=\small, scale=0.7]{$\Gamma_u$};
\draw(1.125,3.25) node[font=\small, scale=0.7]{$\Gamma_d$};
\draw(1.8375,3.25) node[font=\small, scale=0.7]{$\Gamma_u$};
\draw(2.65,3.25) node[font=\small, scale=0.7]{$\Gamma_d$};
\fill (0.75,3) circle (1.2pt);
\fill (1.5,3) circle (1.2pt);
\fill (2.25,3) circle (1.2pt);
\fill (0.75,0) circle (1.2pt);
\fill (1.5,0) circle (1.2pt);
\fill (2.25,0) circle (1.2pt);
\draw(0.35,-0.25) node[font=\small, scale=0.7]{$\Gamma_d$};
\draw(1.125,-0.25) node[font=\small, scale=0.7]{$\Gamma_u$};
\draw(1.8375,-0.25) node[font=\small, scale=0.7]{$\Gamma_d$};
\draw(2.65,-0.25) node[font=\small, scale=0.7]{$\Gamma_u$};
	\end{tikzpicture}
     \label{DesignDomainForPBMixerProblem}

    \end{minipage}%
}
    \hfill
    \subfigure{
    \begin{minipage}[t]{0.45\linewidth}
    \centering
    \begin{tikzpicture}[yscale=0.65,xscale=1.75]
\draw (0,0) rectangle (3,3);
\fill[gray!70] (0.375,0) rectangle (2.625,3);
\node at (1.5,1.5) {\textbf{\textcolor{white}{Design domain}}};
\draw (0,3) .. controls (0.1,1.5) and (0.3,0.6) .. (0,0);
\draw (0.375,3)..controls(0.525,2.7)..(0.75,3);
\draw (0.75,0)..controls(1.125,0.5)..(1.5,0);
\draw (1.5,3)..controls(1.875,2.5)..(2.25,3);
\draw (2.25,0)..controls(2.4,0.3)..(2.625,0);

\fill (0.75,3) circle (0.6pt);
\fill (1.5,3) circle (0.6pt);
\fill (2.25,3) circle (0.6pt);
\fill (0.75,0) circle (0.6pt);
\fill (1.5,0) circle (0.6pt);
\fill (2.25,0) circle (0.6pt);
\tikzset{
  charge/.style={
    circle,
    draw,
    fill=white,
    inner sep=0pt,
    radius=0.1
  }
}
\foreach \y in {2.4,1.5,0.6} {
    \node[charge] at (-0.12,\y) {$-$};
  }
\draw (0.44,3.2) ellipse (0.055 and 0.15);
\node[scale=0.6] at (0.44,3.2) {$-$};
\draw (0.64,3.2) ellipse (0.055 and 0.15);
\node[scale=0.6] at (0.64,3.2) {$-$};

\draw (2.56,-0.2) ellipse (0.055 and 0.15);
\node[scale=0.6] at (2.56,-0.2) {$-$};
\draw (2.36,-0.2) ellipse (0.055 and 0.15);
\node[scale=0.6] at (2.36,-0.2) {$-$};

\draw (1.6,3.2) ellipse (0.055 and 0.15);
\node[scale=0.6] at (1.6,3.2) {$-$};
\draw (1.8,3.2) ellipse (0.055 and 0.15);
\node[scale=0.6] at (1.8,3.2) {$-$};
\draw (2.0,3.2) ellipse (0.055 and 0.15);
\node[scale=0.6] at (2.0,3.2) {$-$};
\draw (2.2,3.2) ellipse (0.055 and 0.15);
\node[scale=0.6] at (2.2,3.2) {$-$};

\draw (1.4,-0.2) ellipse (0.055 and 0.15);
\node[scale=0.6] at (1.4,-0.2) {$-$};
\draw (1.2,-0.2) ellipse (0.055 and 0.15);
\node[scale=0.6] at (1.2,-0.2) {$-$};
\draw (1.0,-0.2) ellipse (0.055 and 0.15);
\node[scale=0.6] at (1.0,-0.2) {$-$};
\draw (0.8,-0.2) ellipse (0.055 and 0.15);
\node[scale=0.6] at (0.8,-0.2) {$-$};

	\end{tikzpicture}
     \label{ElectricityFieldDistribution}

    \end{minipage}%
}
\caption{Design domain (left) and electric field setup (right) for electrokinetic micromixing.}
\label{InitialDesignForTheMixerProblemUnderElectricity}
\end{figure}
\FloatBarrier

\begin{figure}[htbp]
    \centering
    \subfigure{
    \begin{minipage}[t]{0.31\linewidth}
    \centering
    \includegraphics[width=1.15\linewidth]{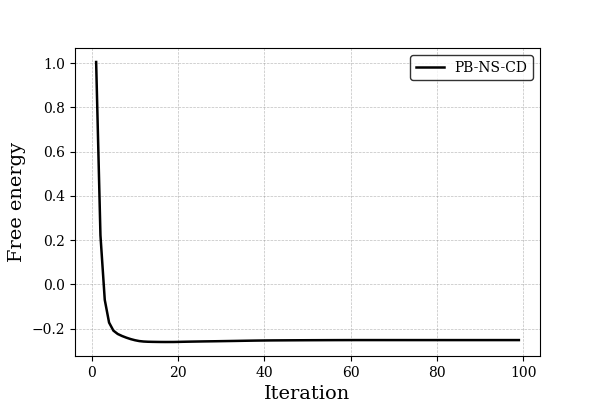}
    \end{minipage}
    }
    \hfill
    \subfigure{
    \begin{minipage}[t]{0.31\linewidth}
    \centering
    \includegraphics[width=1.05\linewidth]{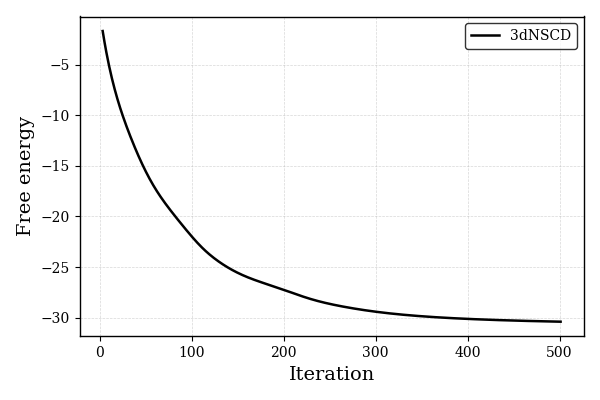}
    \end{minipage}
    }
    \hfill
    \subfigure{
    \begin{minipage}[t]{0.31\linewidth}
    \includegraphics[width=1.05\linewidth]{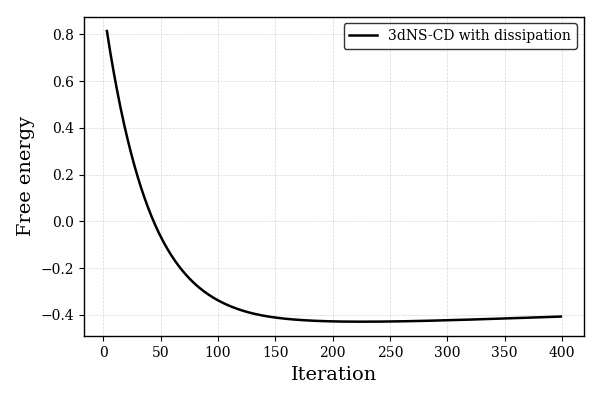}
    \end{minipage}
    }
    \caption{Energy for 2D (left), 3D mixer problem (middle and right) with modified objective.}
    \label{ConvergenceHistoriesOfMixer3d}
\end{figure}

\FloatBarrier
\begin{figure}[htbp]
\centering
    \subfigure[Electricity field for the optimized design]{
    \begin{minipage}[t]{0.45\linewidth}
        \centering
        \includegraphics[width=0.9\textwidth]{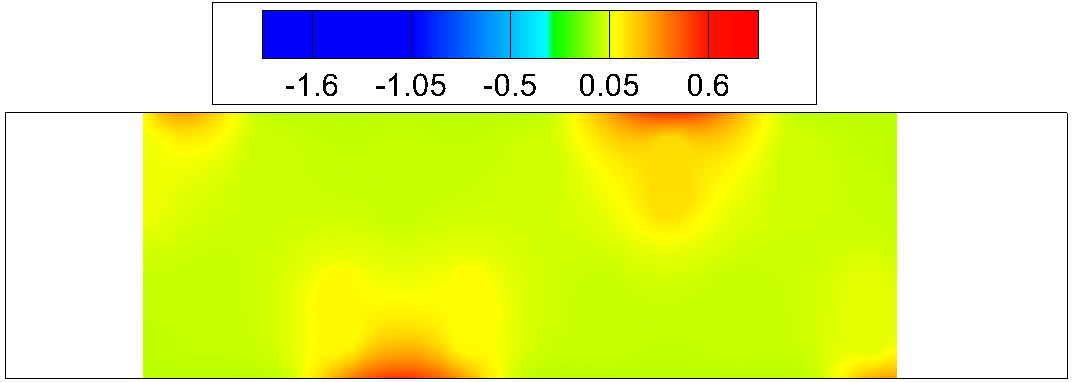}
        \label{PBNSCD2dPsi}
    \end{minipage}%
    }
    \hfill
    \subfigure[Concentration distribution under electricity field]{
    \begin{minipage}[t]{0.45\linewidth}
        \centering
        \includegraphics[width=0.9\textwidth]{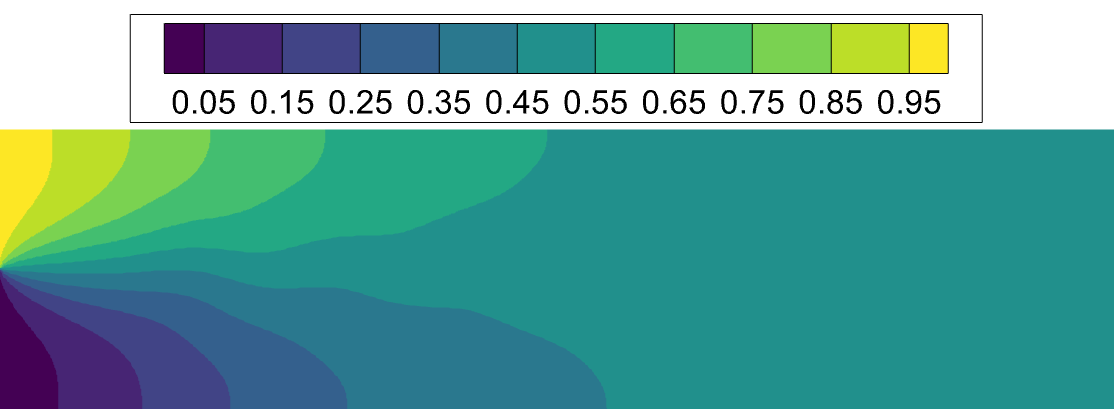}
        \label{PBNSCD2dC}
    \end{minipage}%
    }
  \vspace{1em}
    \subfigure[Optimized design under electricity field]{
    \begin{minipage}[t]{0.45\linewidth}
        \centering
        \includegraphics[width=0.9\textwidth]{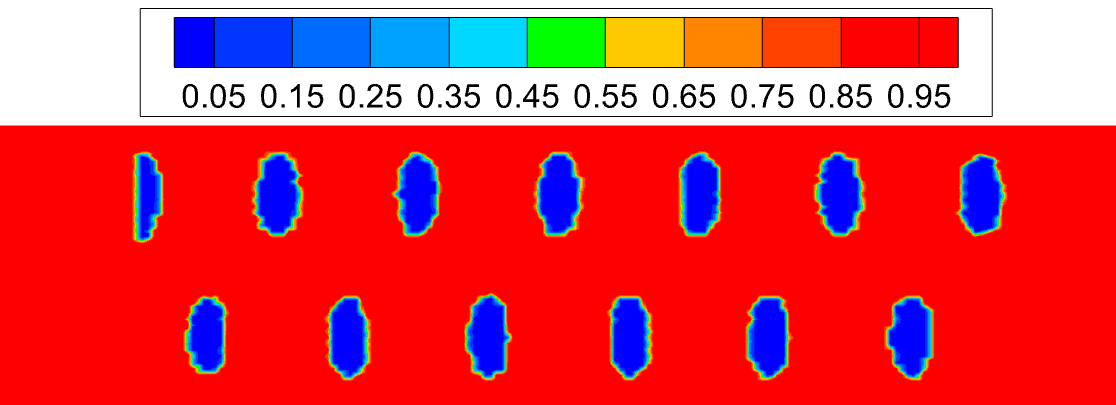}
        \label{PBNSCD2dS}
    \end{minipage}%
    }
    \hfill
    \subfigure[Fluid velocity field under electricity field]{
    \begin{minipage}[t]{0.45\linewidth}
        \centering
        \includegraphics[width=0.9\textwidth]{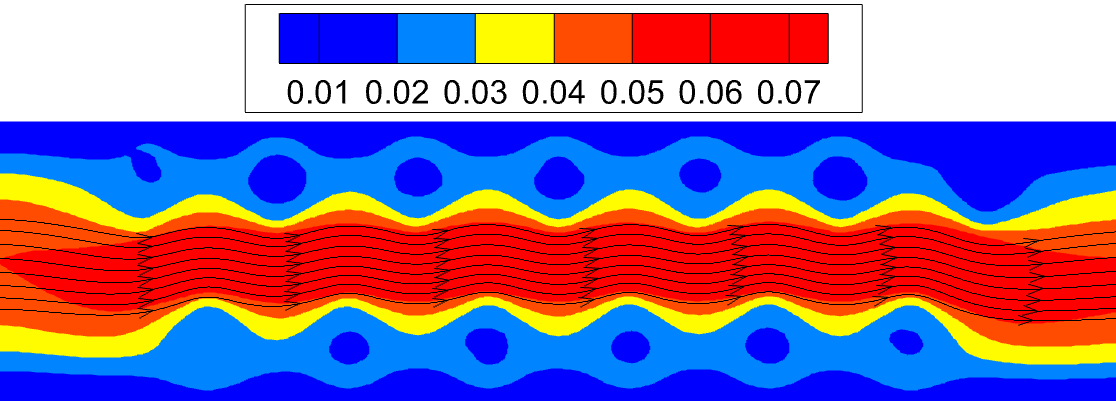}
        \label{PBNSCD2dV}
    \end{minipage}%
    }
\caption{Results on physical fields of the mixer problem under electricity.}
\label{ResultsOfTheMixerProblemUnderElectricity}
\end{figure}

As shown in the Fig. \ref{ResultsOfTheMixerProblemUnderElectricity}, the optimized structure with the electric field is similar to the case without it, where the fluid is still redirected by the solid region. However, by comparing Fig. \ref{NSCD2dV} and Fig. \ref{PBNSCD2dV}, it is evident that the applied electric field helps to adjust the flow velocity, thereby preventing sudden impacts between the fluid and the solid region near the inlet. In this electrokinetically enhanced case, the mixing objective decreases from $3.3 \times 10^{-5}$ to $9.38 \times 10^{-6}$ after $100$ iterations. The corresponding free energy curve also shows the efficiency of the proposed method in handling more complex multiphysics-coupled scenarios.

\subsubsection{3D mixer problem}
Building upon the insights gained from the two-dimensional mixer problem, we now extend the analysis to a three-dimensional setting (see Fig. \ref{DesignDomainFor3dMixerProblem}). This transition not only allows us to assess the scalability and robustness of the proposed optimization framework, but also enables a more realistic investigation of flow behavior in a practical microfluidic system.

\textbf{Example 8 :} For the 3D mixer problem, the same fluid properties are used with ${\rm Re}=1, {\rm Pe}=3\times 10^2, \alpha_0=8 \times 10^2$. We set the inflow velocity to be ${\bm u_0} = [0.1y(1 - y),0,0]^{\rm T}$ and the initial phase-field function 
\begin{equation}
\begin{aligned}
\phi_0(x, y,z) =
\begin{cases}
    1, & \text{if } x < 0.5 \text{ or } x > 3.5, \\
    1 - \left( 0.5 \cos(4\pi x) \cos\left(2\pi y + \frac{\pi}{2} \right) + 0.15 \right), & \text{if } 0.5 \le x \le 3.5,
\end{cases}
\end{aligned}
\end{equation}
and the prescribed concentration distribution as \eqref{InitialConcentration}.

\begin{figure}[htbp]
    \centering
    \subfigure{
    \begin{minipage}[t]{0.45\linewidth}
    \centering

    \includegraphics[width=0.8\textwidth]{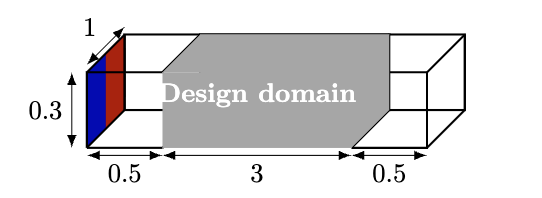}
    \label{DesignDomainFor3dMixerProblem}
    \end{minipage}%
    }
    \hfill
    \subfigure{
    \begin{minipage}[t]{0.45\linewidth}
        \centering
        \includegraphics[width=0.9\textwidth]{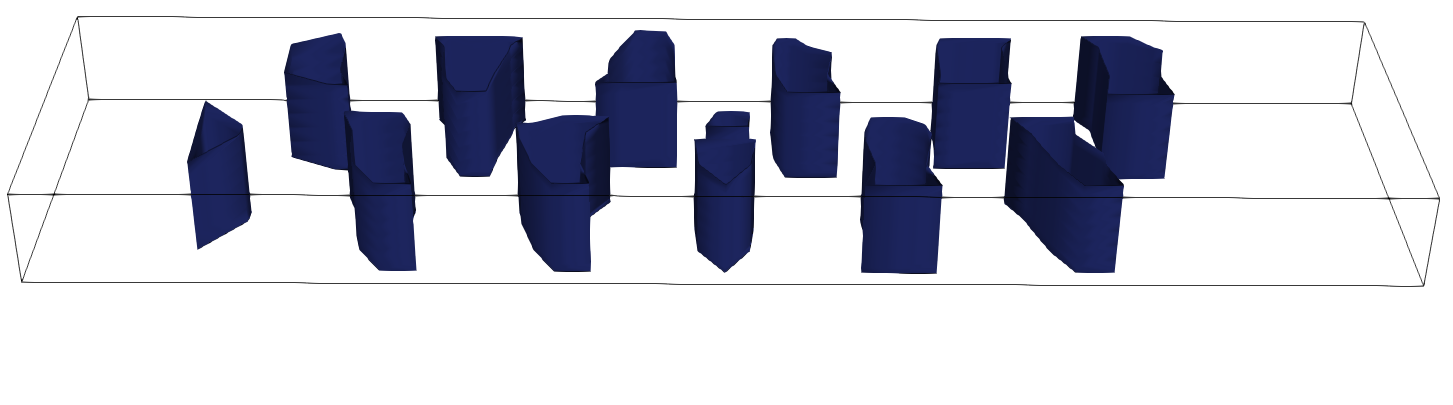}
        \label{InitialShapeMixer3d}
    \end{minipage}%
   } 
\caption{Design domain (left) and initial guess (right) for Example 8 and Example 9.}
\end{figure}
However, due to the significantly increased degrees of freedom in 3D (about $1.2\times 10^5$ in our computation with tetrahedral elements), the phase-field iteration parameters are appropriately adjusted as: ${\tau}=5\times 10^{-6}\ $, $\kappa=0.2\ $, $\eta=10.0\ $, $\beta_1=0$, $\beta_2=1$, $\beta_3=6\times 10^2\ $, $N_p=6$ and $N=5\times 10^2$. The target volume ${ V_0}=0.135$. The optimization process of the phase field distributions is shown in Fig. \ref{PhaseFieldDuringProcess}. The initial voids gradually disappear as the iteration proceeds. The final concentration distribution and velocity slices are shown in Fig. \ref{NSCD3dC} and~\ref{NSCD3dV}, respectively. The evolution of the free energy throughout the optimization is plotted in Fig. \ref{ConvergenceHistoriesOfMixer3d}. As shown in the results, the optimized three-dimensional structure resembles a vertical stacking of the two-dimensional design, reflecting consistent flow behavior across slices. 
Both the free energy and volume error exhibit smooth and monotonic decay, with the final volume error remaining within $3\%$ of the target, as shown in Fig. \ref{ConvergenceHistoriesOfMixer3d} (middle).

\begin{figure}[htbp]
    \centering
    \subfigure[Iteration for 200 steps]{
    \begin{minipage}[t]{0.45\linewidth}
        \centering
        \includegraphics[width=0.9\textwidth]{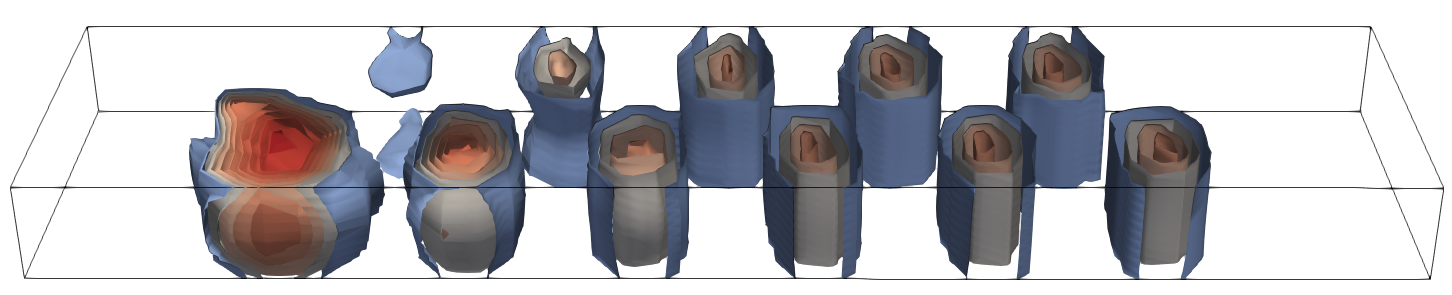}
        \label{NSCD3dS(200)}
    \end{minipage}
    }
    \hfill
    \subfigure[The final optimized phase-field distribution]{
    \begin{minipage}[t]{0.45\linewidth}
        \centering
        \includegraphics[width=0.9\textwidth]{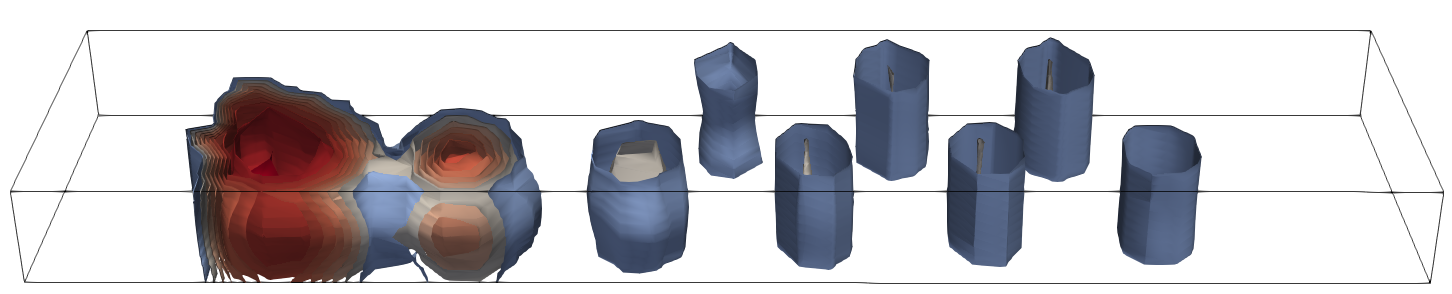}
        \label{NSCD3dS}
    \end{minipage}
    }
    \hfill
        \subfigure[Concentration distribution on final design]{
    \begin{minipage}[t]{0.45\linewidth}
        \centering
        \includegraphics[width=0.9\textwidth]{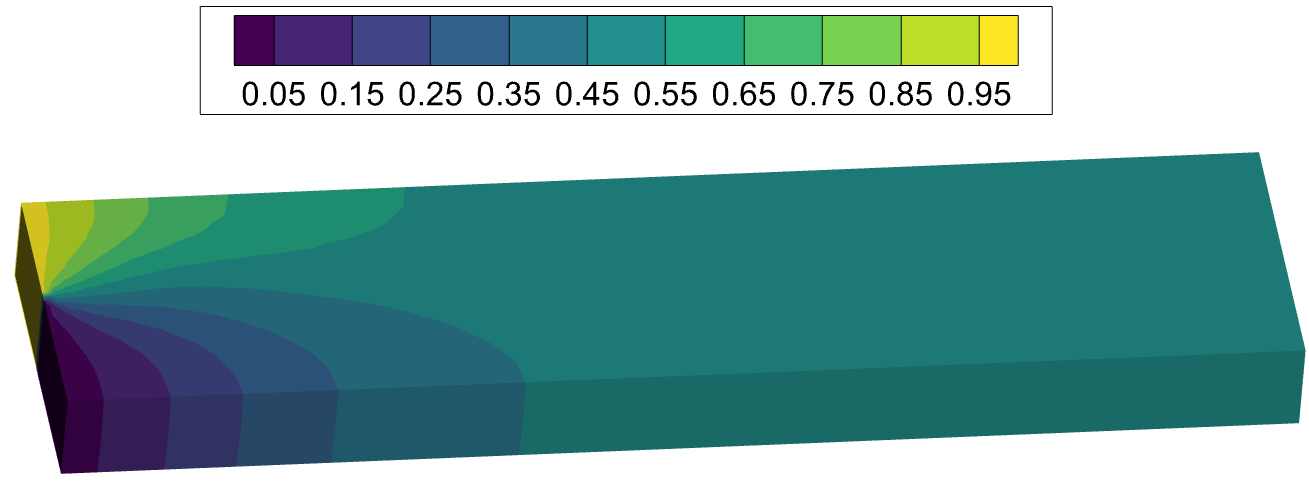}
        \label{NSCD3dC}
    \end{minipage}%
    }
    \hfill
    \subfigure[Velocity field on final design]{
    \begin{minipage}[t]{0.45\linewidth}
        \centering
        \includegraphics[width=0.9\textwidth]{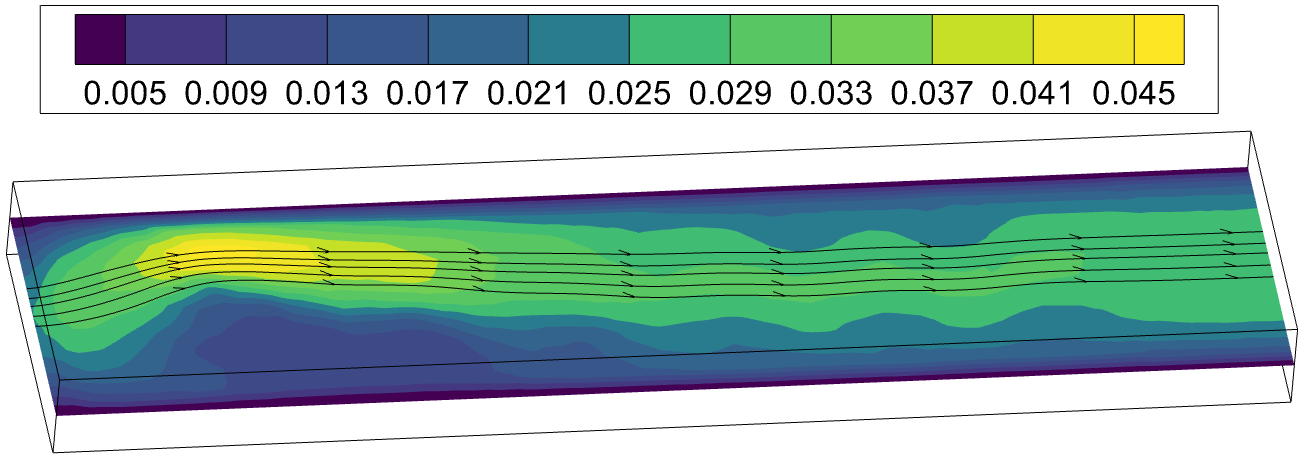}
        \label{NSCD3dV}
    \end{minipage}
    }
    \caption{Design process for the 3D mixer problem: Example 8.}
    \label{PhaseFieldDuringProcess}
\end{figure}


However, the final phase-field distribution reveals the accumulation of solid material near the inlet, leading to partial blockage, a phenomenon commonly observed in microchannel mixing problems. To address this, previous works (e.g. \cite{Alexandersen2020_ReviewFluidTopology} ) have proposed strategies such as spatial filtering and additional pressure drop constraints. In this study, we introduce a new approach aimed at mitigating inlet blockage within the topology optimization framework.

\FloatBarrier
\textbf{Example 9 :}
The occurrence of inlet blockage in micromixing problems can be interpreted as a consequence of the system’s attempt to enhance mixing by concentrating solutes into a confined region, thereby promoting passive diffusion. In previous studies, this issue has been addressed by adding pressure drop constraints to the objective, which essentially encourages smoother tangential flow along the channel. Inspired by this idea, we propose an alternative multi-objective formulation: in addition to the original mixing objective, we introduce a small dissipation-related term. Since blockage near the inlet results in high local resistance, this dissipation term promotes smoother flow throughout the domain and effectively reduces the likelihood of blockage. We adopt the same initial phase-field function and velocity as in \textbf{Example 8}. However, in this case, we perform a bi-objective optimization by adding a small dissipation term $\beta_1 J_1$ to the original objective function and set $N=4\times 10^2$.

The free energy functional applied in the modified Allen-Cahn equation is defined as
\eqref{W3}, where $\beta_1 = 10^{-2}$ and $\beta_2=1$.
The optimization process is shown in Fig. \ref{PhaseFieldWithDissiDuringProcess}. The evolution of free energy throughout the iteration is plotted in Fig. \ref{ConvergenceHistoriesOfMixer3d} (right), where the inclusion of the dissipation term significantly alleviates the inlet blockage. The optimized structure in Fig. \ref{NS3dWithDissiS} indicates that the fluid is continuously redirected by the solid regions and primarily exits through the central part of the channel, which is a natural extension behavior consistent with the two-dimensional case. 
 
\begin{figure}[htbp]
    \centering
    \subfigure[Design after 200 steps]{
    \begin{minipage}[t]{0.45\linewidth}
        \centering
        \includegraphics[width=0.9\textwidth]{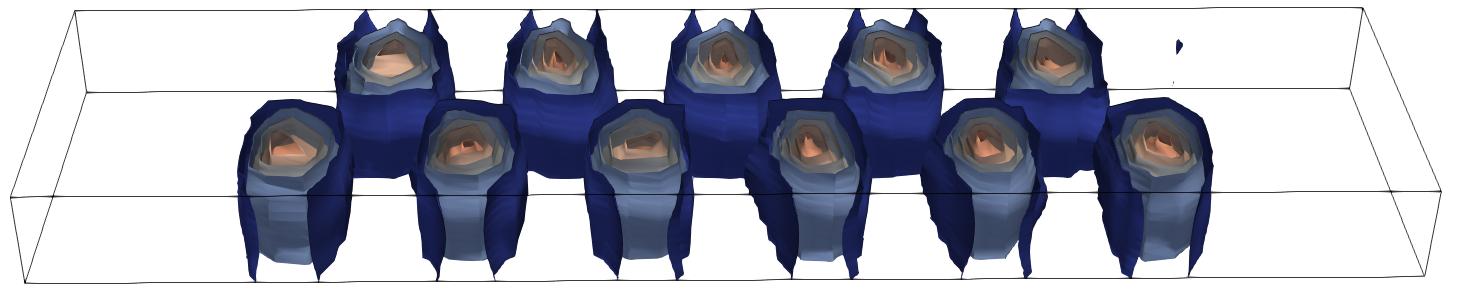}
        \label{NS3dWithDissiS(200)}
    \end{minipage}
    }
    \hfill
    \subfigure[Final optimized design]{
    \begin{minipage}[t]{0.45\linewidth}
        \centering
        \includegraphics[width=0.9\textwidth]{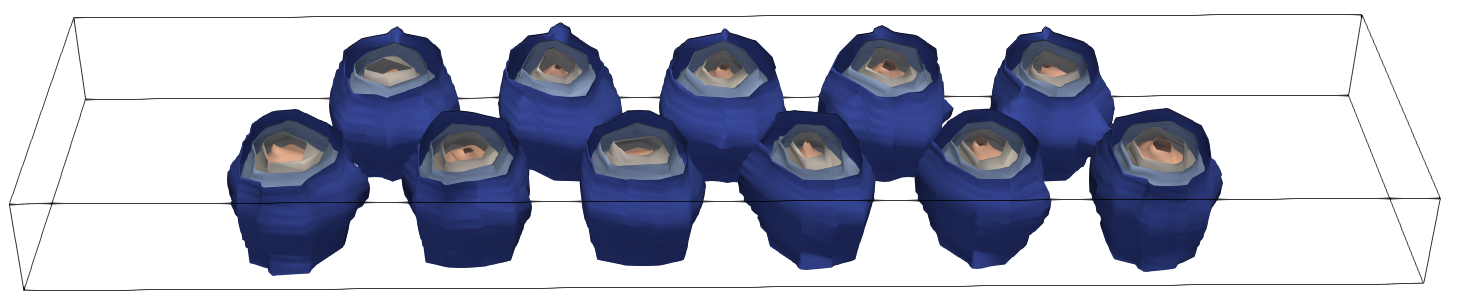}
        \label{NS3dWithDissiS}
    \end{minipage}
    }
    \hfill
    \subfigure[Concentration distribution on the optimized design]{
    \begin{minipage}[t]{0.45\linewidth}
        \centering
        \includegraphics[width=0.9\textwidth]{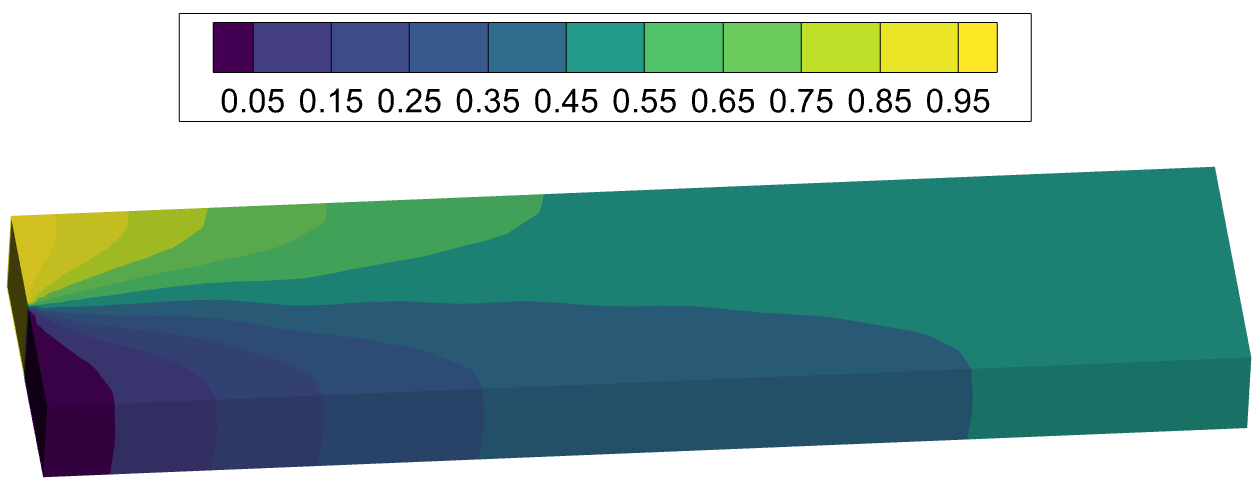}
        \label{NS3dWithDissiC}
    \end{minipage}%
    }
    \hfill
    \subfigure[Velocity field on the optimized design]{
    \begin{minipage}[t]{0.45\linewidth}
        \centering
        \includegraphics[width=0.9\textwidth]{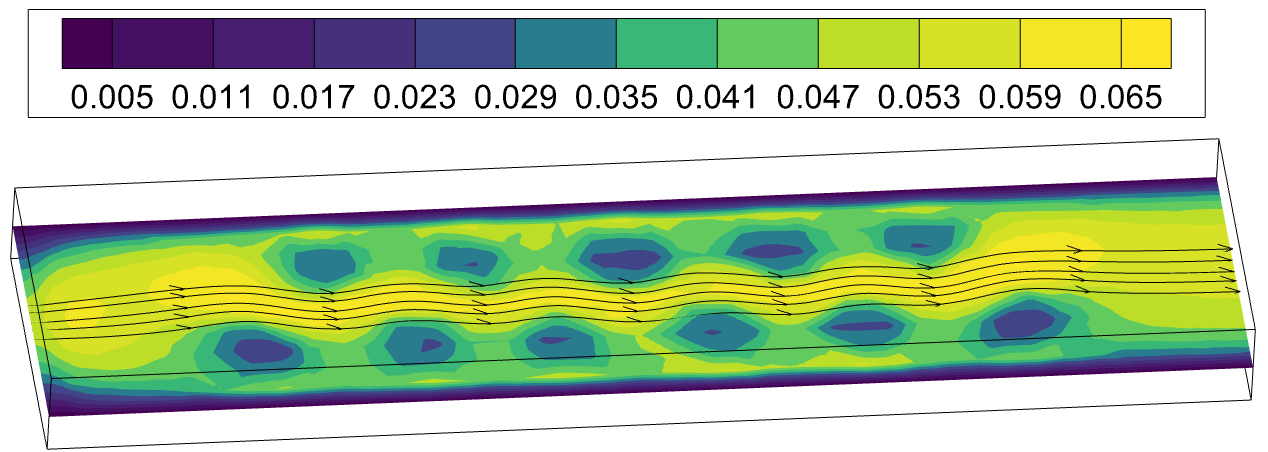}
        \label{NS3dWithDissiV}
    \end{minipage}
    }
    \caption{Results for the 3D mixer problem: Example 9.}
    \label{PhaseFieldWithDissiDuringProcess}
\end{figure}



\section{Conclusion} \label{conclusion}
In this work, we coupled the fluid system with the convection-diffusion equation and the Poisson–Boltzmann equation to construct a new multiphysics model for shape design. Starting from the physical interpretation of the phase-field equation, we derived the free energy expression based on the Ising model and thereby corrected the original Ginzburg–Landau formulation. The stability and effectiveness of the proposed numerical scheme were demonstrated through energy dissipation problems. We further validated the robustness of the method under multiphysics coupling by applying it to two-dimensional benchmark micromixing problems, including both classical and electro-kinetically driven cases. To address the frequent inlet blockage observed in three-dimensional configurations, we proposed a multi-objective formulation incorporating a dissipation-based term as a mitigation strategy. We conclude that, owing to its preservation of the underlying physical properties, the proposed free energy formulation holds strong potential for broader applications in more complex multiphysics coupling models.

\appendix
\section{Derivation of the Ginzburg–Landau free energy from the Ising Model}
\label{appendix:IsingGL}
In the classical Ising model, each spin interacts only with its nearest neighbors. Under the assumption of coarse-graining, the domain is partitioned into small subregions (or blocks), and the local behavior within each block is characterized by an average spin value. The energy of a given block includes contributions from nearest-neighbor interactions as well as coupling with an external field. To provide a clearer understanding of the microscopic interaction mechanism, a schematic diagram of the Ising model is presented in Fig. \ref{Ising}. 
\begin{figure}[htbp]
    \centering
    \includegraphics[width=0.6\linewidth]{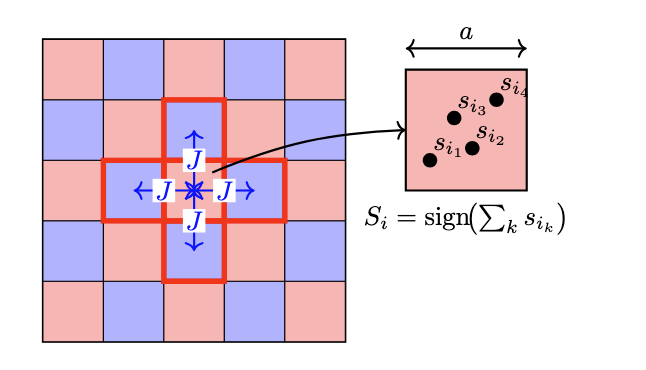}
    \caption{Ising model}
    \label{Ising}
\end{figure}

The domain is discretized into a regular lattice, where each small block (or coarse-grained cell) with length $a\times a$ is indexed by~$i$. Within the $i$-th block, individual particles are denoted by $s_{i_k}$, taking values in $\{+1, -1\}$. The coarse-grained spin variable $S_i$ associated with the $i$-th block is defined by the majority rule
\begin{equation*}
S_i = \mathrm{sign}\ ( \sum_k s_{i_k} ).
\end{equation*}

As illustrated in Fig. \ref{Ising}, each block interacts with its nearest neighbors. Under an external field $h$, the free energy of the system can be expressed as in \cite[Eq. (2.95)]{Goldenfeld1992}:
\begin{equation*}
F = -\tilde{J} \sum_{\langle i, j \rangle} S_i S_j - \sum_i h S_i,
\end{equation*}
where $\tilde{J} > 0$ represents the interaction strength promoting alignment between neighboring spins, $\langle i, j \rangle$ denotes unordered nearest-neighbor pairs with each pair operated on only once. The corresponding partition function reads:
\begin{equation*}
Z = \sum_{\{S_i\}} \exp( -\frac{F}{k_B T} ),
\end{equation*}
where $k_B$ is the Boltzmann constant, $T$ is the absolute temperature, and the summation is taken over all possible spin configurations $\{S_i\}$. By applying the Gaussian integral identity
\begin{equation*}
e^{\frac{a^2}{2}} = \frac{1}{\sqrt{2\pi}} \int_{-\infty}^{\infty} \exp( -\frac{x^2}{2} + a x ) \dx,
\end{equation*}
and exchanging the order of summation and integration, the partition function can be reformulated as
\begin{equation*}
\begin{aligned}
Z &= \sum_{\{S_i\}} \exp[\frac{1}{k_B T}(\tilde{J}\sum_{\langle i,j\rangle}S_i S_j+\sum_{i}hS_i)]\\
&=\sum_{\{S_i\}} \exp[\frac{1}{2k_B T}(\tilde{J}(\sum_{\langle i,j\rangle}((S_i+S_j)^2-2)+\sum_{i}hS_i)]\\
&=\frac{1}{\sqrt{2\pi}}\sum_{\{S_i\}}\int\prod_{\langle i,j\rangle} \exp[{-\frac{\phi_{ij}^2}{2}-\frac{\tilde{J}}{k_BT}+\sqrt{\frac{\tilde{J}}{k_BT}}\phi_{ij}(S_i+S_j)]}\prod_i\exp(\frac{hS_i}{k_BT})  \mathcal{D}\phi\\
&=\frac{1}{\sqrt{2\pi}}\sum_{\{S_i\}}\int\exp\bigg[\sum_{{\langle i,j\rangle}}(-\frac{\phi_{ij}^2}{2}-\frac{\tilde{J}}{k_BT})+\sum_i\sqrt{\frac{\tilde{J}}{k_BT}}S_i\sum_{j\in\partial i}\phi_{ij}+\frac{hS_i}{k_BT}\bigg]\mathcal{D}\phi\\
&=\frac{1}{\sqrt{2\pi}}\int\exp[\sum_{{\langle i,j\rangle}}(-\frac{\phi_{ij}^2}{2}-\frac{\tilde{J}}{k_BT})+\sum_i\ln[2\cosh(\sqrt{\frac{\tilde{J}}{k_BT}}\sum_{j\in\partial i}\phi_{ij}+\frac{h}{k_BT})]]\mathcal{D}\phi
\end{aligned}
\end{equation*}
where $\int (\cdot)\mathcal{D}\phi$ denotes the functional integration over the auxiliary bond field $\phi_{ij}$ introduced to decouple the quadratic spin–spin interaction, and $\partial i$ denotes the neighbors of particle $i$ (excluding $i$ itself).
It should be emphasized that $\phi_{ij}$ encodes the interaction between two neighboring particles, and a proper description of its contribution at each site is indispensable. Hence, we introduce a continuous order parameter field by averaging over adjacent bonds, $\phi_i = \frac{1}{z} \sum_{j\in\partial i} \phi_{ij}$, where $z$ is the coordination number. In the continuum limit, the bond field can equivalently be expressed as the difference of site fields: $\phi_{ij}=\phi_j-\phi_i$. For a lattice constant $a$, a particle located at $i+a\bm e$ in the direction $\bm e$ relative to site $i$ satisfies
\begin{equation*}
    \phi_j=\phi_i+a \bm{e}\cdot\nabla\phi_i+\frac{a^2}{2}(\bm{e}\cdot\nabla)^2\phi_i+o(a^2).
\end{equation*}
Multiplying both sides by $\phi_i$ and summing over $i$, then performing a long-wavelength expansion, 
\begin{equation*}
    \sum_{\langle i,j\rangle}\phi_i\phi_j=\sum_i(z\phi_i^2-\frac{a^2}{2}|\nabla\phi_i|^2)+o(a^2),
\end{equation*}
so that 
\begin{equation*}
    \sum_{\langle i,j\rangle}(-\frac{\phi_{ij}^2}{2})=-\frac{1}{2}\sum_{\langle i,j\rangle}(\phi_j-\phi_i)^2=\sum_i(-\frac{z}{2}\phi_i^2-\frac{a^2}{4}|\nabla\phi_i|^2).
\end{equation*}
Set $C_0(T):=\frac{1}{\sqrt{2\pi}}\int\exp[\sum_{{\langle i,j\rangle}}(-\frac{\tilde{J}}{k_BT})]\mathcal{D}\phi$. An approximate expression for the partition function can be expressed as:
\begin{equation*}
Z = C_0(T)\int\exp[\sum_i(-\frac{z}{2}\phi_i^2-\frac{a^2}{4}|\nabla\phi_i|^2+\ln[2\cosh(z\sqrt{\frac{\tilde{J}}{k_BT}}\phi_i+\frac{h}{k_BT})])]\mathcal{D}\phi.
\end{equation*}

Consider the case without the external field $h$, using the thermodynamic relation for the Helmholtz free energy 
$F = -k_B T \ln Z$, and apply the Taylor expansion at $\phi_i=0$, we obtain the Ginzburg--Landau free energy density functional $\mathcal{F}(\phi)$ in the form:
\begin{equation*}
\mathcal{F}(\phi) =\frac{\kappa}{2}|\nabla\phi|^2+\hat{w}(\phi)+o(\phi^4),
\end{equation*}
where $\kappa:=\frac{a^2 k_B T}{2}$ and the local free energy $\hat{w}(\phi):=A(T)\phi^2+B(T)\phi^4$. The standard Ginzburg-Landau free energy can be derived as
\begin{equation*}
    F(\phi) =\int_\Omega \mathcal{F}(\phi)\dx=\int_{\Omega}\frac{\kappa}{2}|\nabla\phi|^2+\hat{w}(\phi)\dx,
\end{equation*}
which takes the same form as in \cite[Eq. (5.21)]{Goldenfeld1992}.
However, as seen from the expression above, when the temperature drops below the critical value
$
T_c := \frac{z\tilde{J}}{k_B}$,
the quadratic coefficient $A(T)$ in  $w(\phi)$ becomes negative, and the local free energy density $\hat{w}(\phi)$ undergoes a transition from a single-well to a double-well potential. 

\bibliographystyle{plain} 
\bibliography{references}

\end{document}